\documentclass[10pt, reqno]{article}
\usepackage[french, english]{babel}
\usepackage{amsmath}
\usepackage{hyperref}
\usepackage{amsthm}
\usepackage{pdfsync}
\usepackage{a4wide}
\usepackage{amssymb}
\usepackage{amsfonts}
\usepackage{color}
\usepackage{enumerate}
\usepackage{graphicx}
\usepackage{amsfonts}
\usepackage{color}
\usepackage{enumerate}
\newtheorem{Th}{Theorem}

\newtheorem{prop}{Proposition}
\newtheorem{Def}{Definition}
\newtheorem{obs}{Remark}
\newtheorem{lema}{Lemma}
\def\rr{\mathbb{R}}

\def\intos{\int_{\Omega_{r_0}}}

\def\into{\int_{\Omega}}
\def\hoi{H_{0}^{1}(\Omega)}

\def\la{\lambda}
\def\n{\nabla}
\def\D{\Delta}
\def\q{\quad}

\def\intot{\iint\limits_{Q_T}}

\def\p{\partial}

\def\intot{\iint\limits_{Q_T}}

\def\eps{\varepsilon}

\def\eps{\varepsilon}

\def\into{\int_{\Omega}}
\def\hoi{H_{0}^{1}(\Omega)}

\def\d{\partial}
\def\eps{\varepsilon}

\numberwithin{equation}{section} \numberwithin{Th}{section}
\numberwithin{cor}{section} \numberwithin{lema}{section}
\numberwithin{prop}{section} \numberwithin{obs}{section}
\numberwithin{Def}{section}

\begin{document}
\title{Schr\"{o}dinger operators with boundary singularities: Hardy
inequality, Pohozaev identity and controllability results
 }
\author{Cristian Cazacu
 \footnote{BCAM - Basque Center for
Applied Mathematics, Bizkaia Technology Park 500, 48160, Derio,
Basque Country, Spain}  \footnote{Departamento de Matem\'{a}ticas,
Universidad Aut\'{o}noma de Madrid,
 Madrid 28049, Spain.} }

\maketitle
\begin{abstract}
The aim  of this paper is two folded. Firstly, we study the validity
of the Pohozaev-type identity for the Schr\"{o}dinger
 operator
$$A_\la:=-\D -\frac{\la}{|x|^2}, \q \la\in \rr,$$
in the situation where the origin is located on the boundary  of a
smooth domain $\Omega\subset \rr^N$, $N\geq 1$. The problem we
address is very much related to optimal Hardy-Poincar\'{e}
inequality with boundary singularities which has been investigated
in the recent past in various papers. In view of that,  the proper
functional framework is described and explained.

Secondly, we apply the Pohozaev identity not only to study
semi-linear elliptic equations but also to derive the method of
multipliers in order to study the exact boundary controllability of
the wave and Schr\"{o}dinger equations corresponding to the singular
operator $A_\la$. In particular, this complements and extends well
known results by Vanconstenoble and Zuazua \cite{judith}, who
discussed the same issue in the case of interior singularity.
\end{abstract}
\tableofcontents 

\section{Introduction}
In this paper we are dealing with the Schr\"{o}dinger operator
$A_{\la}:=-\D-\la/|x|^2$, $\la \in \rr$,
 acting in a domain where the potential $1/|x|^2$ is singular at the
boundary. Our main goal consists to study the control properties of
the corresponding wave and Schr\"{o}dinger equations. Moreover, we
are aimed to give necessary and sufficient conditions for the
existence of non-trivial solutions to semi-linear elliptic equations
associated to $A_\la$.
     Operators like $A_\la$ may arise in molecular
physics \cite{pot1}, quantum cosmology \cite{pot2}, combustion
models \cite{MR1616905} but also in the linearization of critical
nonlinear PDE's playing a crucial role in the asymptotic behaviour
of branches of solutions in bifurcation problems (e.g.
\cite{MR1605678}, \cite{MR1616516}). From the mathematical view
point they are interesting due to their criticality since  they are
homogeneous of degree -2.

The  qualitative properties of evolution problems involving the
operator $A_\la$ require either positivity or coercivity of $A_\la$
in the sense of quadratic forms in $L^2$. Roughly speaking, this is
equivalent to make use of Hardy-type inequalities. There is a large
literature concerning the study of such inequalities, especially in
the context of interior singularities (e.g. see \cite{vazzua},
\cite{adimurthi1}, \cite{MR2099546} and references therein). The
classical Hardy inequality  is stated as follows. Assume $\Omega$
 is a smooth bounded domain in $\rr^N$, $N\geq 3$,
  containing the origin, i.e.,
 $0\in \Omega$; then it follows
(see \cite{hardy-polya})
 \begin{equation}\label{Hardy}
 \into |\n u|^2 dx-\frac{(N-2)^2}{4}\into \frac{u^2}{|x|^2}dx>0, \q
 \forall u\in \hoi,
 \end{equation}
 and the constant $(N-2)^2/4$ is optimal and not attained.
  We remind that the optimal
Hardy constant is defined by the quotient
$$\mu(\Omega):=\inf_{u\in C_{0}^{\infty}(\Omega)}
 \Big(\into |\n u|^2 dx \big /\into u^2/|x|^2 dx\Big).$$
 In this paper, we
consider $\Omega$ to be a smooth subset of $\rr^N$, $N\geq 1$, with
the origin $x=0$ placed on its boundary $\Gamma$. Hardy inequalities
with an isolated singularity on the boundary were less investigated
so far. However, in the recent past some substantial work has been
developed in that direction.

 It has been proved that,
the best constants depends both on the local geometry near the
origin and the entire shape of the domain.

More precisely,  starting with the work by Filippas, Tertikas and
Tidblom \cite{terti}, and continuing with \cite{cristiCRAS},
\cite{Fall1}, \cite{musina}, it has been proved that,
 whenever $\Omega$ is a
smooth domain with the origin located on the boundary, there exists
a positive constant $r_0=r_0(\Omega, N)>0$ such that
\begin{equation}\label{best}
\mu(\Omega\cap B_{r_0}(0))=\frac{N^2}{4}.
\end{equation}
where $B_{r_0}(0)$ denotes the $N$-d ball of radius $r_0$ centered
at origin. Next  we recall the definition of  the upper half space
$\rr_{+}^N$ which is given by the set $\rr_{+}^{N}:=\{x=(x_1,
\ldots, x_{N-1}, x_N)=(x', x_N)\in \rr^{N-1}\times \rr \ | \ x_N>0
\}$.
 In addition, if
$\Omega\subset\rr_{+}^{N}$, $N\geq 1$, the new Hardy inequality
\begin{equation}\label{equ1}
\into|\nabla u|^2dx\geq \frac{N^2}{4}\into\frac{u^2}{|x|^2}dx \quad
\forall\quad u\in \hoi.
\end{equation}
holds true and the constant $N^2/4$ is optimal  i.e.
$\mu(\Omega)=N^2/4$.

Otherwise, if $\Omega$ is a smooth domain which, up to a rotation,
is not supported in $\rr_{+}^{N}$,  the constant $N^2/4$ is optimal,
up to lower order terms in $L^2(\Omega)$-norm as shown later in
inequality
 \eqref{oeq44}. In general $\mu(\Omega)=N^2/4$ is not
true for any smooth bounded domain $\Omega$ containing  the origin
on the boundary (e.g. \cite{Fall1}).

Without losing the generality, since the operator $A_\la$ is
invariant under rotations,  next we consider $\Omega$ such that
\begin{equation}\label{cond}
x\cdot \nu = O(|x|^2), \textrm{ on } \Gamma,
\end{equation}
where $\nu$ stands for the outward normal vector to $\Gamma$.
%
 Moreover, since optimal inequalities have been
obtained regardless the shape of $\Omega$, throughout the paper we
discuss two main
situations of geometries motivated by the remarks above.

\begin{enumerate}[C1.]
\item\label{eqq2} 
 $\Omega$ is a smooth domain satisfying \eqref{cond} and $x_N>0$ holds for
 all $x\in \Omega$ (i.e.
$\Omega\subset
\rr_{+}^{N}$).  
 \item\label{eqq3} $\Omega$ is a smooth domain satisfying \eqref{cond}
 such that $x_N$ changes sign in $\Omega$  $(\Omega \not \subset \rr_{+}^{N})$.
\end{enumerate}
Next we need to introduce the constant
\begin{equation}\label{diam}
R_\Omega=\sup_{x\in \overline{\Omega}}|x|.
\end{equation}
 The following optimal Hardy-Poincar\'{e} inequalities
are valid for each one of the cases above.

If $\Omega$ fulfills the case C1, then (e.g. \cite{cristiCRAS})  it
holds that
\begin{equation}\label{oeq3}
\forall \ u\in C_{0}^{\infty}(\Omega), \q \into |\nabla u|^2dx\geq
\frac{N^2}{4}\into \frac{u^2}{|x|^2}dx+\frac{1}{4}\into
\frac{u^2}{|x|^2\log^2(R_{\Omega}/|x|)}dx,
\end{equation}
and $N^2/4$ is the sharp constant.

If $\Omega$ satisfies the case C2 then (e.g. \cite{Fall1}) there
exist two constants $C_2=C_2(\Omega)\in \rr$ and $C_3=C_3(\Omega,
N)>0$ such that  for any $u\in C_{0}^{\infty}(\Omega)$ it holds
\begin{equation}\label{oeq44}
C_2\into u^2dx +\into |\nabla u|^2dx\geq \frac{N^2}{4}\into
\frac{u^2}{|x|^2}dx +C_3\into
\frac{u^2}{|x|^2\log^2(R_\Omega/|x|)}dx.
\end{equation}
  In view of these, let us  now describe the content of the paper.

 In  Section \ref{secpoho}, we  firstly introduce the
functional framework induced by the above Hardy inequalities. We
refer to the Hilbert space $H_\la$ defined in Subsection
\ref{1subsec}. Then we check the validity of the Pohozaev identity
for the Schr\"{o}dinger operator $A_\la$ in this functional setting
as follows. The domain of $A_\la$ is defined by
\begin{equation}\label{domeniu}
D(A_\la):=\{u\in H_\la \ |\ A_\la u \in L^2(\Omega)\},
\end{equation}
and it  holds
\begin{align}\label{Pohozaevintro}
\frac{1}{2}\int_{\Gamma}  (x\cdot \nu) \Big(\frac{\p u}{\p
\nu}\Big)^2d\sigma=-\into (x\cdot \n u) A_\la u  dx-
\frac{N-2}{2}||u||_{H_\la}^2, \q \forall\  u\in D(A_\la),
\end{align}
where $||\cdot||_{H_\la}$ denotes the norm associated to $H_\la$ and
We refer to Theorems \ref{trace}, \ref{poho} for a complete
statement of this result.
 For the sake of clarity,  we will mainly discuss the case C1 above.
  Nevertheless, similar  results could be also extended to the case C2 in a
weaker functional setting due to   weaker Hardy inequalities (see
Subsection \ref{2sec}).

Formally, identity \eqref{Pohozaevintro} can be  obtained by direct
integrations. However, this is not rigorously allowed because the
lack of regularity of $A_\la$ at the origin, otherwise the potential
$1/|x|^2$ is bounded and the standard elliptic regularity applies.
In addition, we need to justify the integrability of the boundary
term in \eqref{Pohozaevintro} which is no more obvious since the
singularity is located on the boundary and  standard trace
regularity does not applies. As we mentioned before, we give a
rigorous justification of these facts in Theorems \ref{trace},
\ref{poho}.

Pohozaev type identities arise in many applications and mostly when
studying non-linear equations (see \cite{MR1625845},
\cite{MR2097030}, \cite{peraldavila} and references therein).

 In Section \ref{8sec}, we apply Theorem \ref{poho} to
 characterize the existence
 of non-trivial solutions
  to a semi-linear singular elliptic PDE  in star-shaped
domains. We refer mainly to Theorem \ref{th1}.

In Section \ref{5sec} we present some applications of Theorem
\ref{poho} in Controllability of conservative systems like wave and
Schr\"{o}dinger equations, for which the multiplier method plays a
crucial role.

 In the last few decades,  most
of the studies in Controllability Theory and its applications to
evolution PDEs,  have applied methods like \textit{Hilbert
Uniqueness Method} (HUM) introduced by J. L. Lions in \cite{lions1},
Carleman estimates developed by Fursikov and Imanuvilov
\cite{fursikov}, microlocal analysis due to Bardos, Lebeau and Rauch
(\cite{BardosLeRa}, \cite{BardosLeRa2}), but also multiplier
techniques with the pioneering papers by Komornik and Zuazua
(\cite{komornik}, \cite{zuazua2}, \cite{zuazua3}). In particular,
the controllability properties and stabilization of the heat like
equation corresponding to $A_\la$ have been analyzed in
\cite{heatjudith}, \cite{sylvain}, \cite{vanc} in the case of
interior singularity using tools like multiplier techniques and
Carleman estimates.

Now, let us detail the problem we are interested in Section
\ref{5sec}.  For $N\geq 1$ we consider  a bounded smooth domain
$\Omega\subset \rr^N$ where $\Gamma$ denotes its boundary. Moreover,
we state by $\Gamma_0$ a non-empty part of the set $\Gamma$ that
will be precise later.

Next we consider the Wave-like process
\begin{equation}\label{eq124}\left\{\begin{array}{ll}
  u_{tt}-\Delta u-\lambda \frac{u}{|x|^2}=0, & (t,x)\in (0, T)\times \Omega, \\
  u(t,x)=h(t,x), & (t,x)\in (0,T)\times \Gamma_0, \\
  u(t,x)=0, & (t,x)\in (0,T)\times (\Gamma\setminus \Gamma_0), \\
  u(0,x)=u_0(x), &  x\in \Omega, \\
  u_t(0,x)=u_1(x), & x \in \Omega. \\
\end{array}\right.
\end{equation}

 To make the problem under consideration
precise  we say that the system (\ref{eq124}) is exactly
controllable from $\Gamma_0$, in time $T$, if for any initial data
$(u_0,u_1)\in L^2(\Omega)\times H_{\la}^{'}$ and any target
$(\overline{u_0},\overline{u_1})\in L^2(\Omega)\times H_{\la}^{'}$ ,
there exists a control $h\in L^2((0,T)\times \Gamma_0)$ such that
the solution of (\ref{eq124}) satisfies:
    $$(u_t(T,x),u(T,x))=(\overline{u_1}(x),\overline{u_0}(x)) \quad \textrm{ for all } x\in \Omega. $$

This issue was analyzed by Vancostenoble and Zuazua \cite{judith}
under the assumption that  the singularity  $x=0$ is  located in the
interior of $\Omega$.  They proved well-posedness and exact
controllability of the system (\ref{eq124}) for any $\la\leq
\la_\star:=(N-2)^2/4$ from the boundary observability region
$\Gamma_0$ described by
\begin{equation}\label{oeq2}
\Gamma_0:=\{x\in \Gamma\ |\ x \cdot \nu \geq 0\}.
\end{equation}
Roughly speaking, the authors showed in \cite{judith} that the
parameter $\lambda_\star$ is critical when asking the well-posedness
and control properties of (\ref{eq124}), and the results are very
much related to the best constant in the Hardy inequality with
interior singularity.

In Section \ref{5sec},  we address the same controllability question
in the case of boundary singularity. Our main result asserts that
for the same geometrical setup \eqref{oeq2}, we can increase the
range of values $\lambda$ (from $\la_\star$ to $\la(N):=N^2/4$) for
which the exact boundary controllability of system \eqref{eq124}
holds. This is due to the new Hardy inequalities above.

 By now classical HUM, the Controllability of system
(\ref{eq124}) is equivalent to so-called \textsl{Observability
Inequality} for the adjoint system,
\begin{equation}\label{eq2}\left\{\begin{array}{ll}
  v_{tt}-\Delta v-\lambda \frac{v}{|x|^2}=0, & (t,x)\in (0, T)\times \Omega, \\
  v(t,x)=0, & (t,x)\in (0,T)\times \Gamma, \\
  v(0,x)=v_0(x), &  x\in \Omega, \\
  v_t(0,x)=v_1(x), & x \in \Omega, \\
\end{array}\right.
\end{equation}
which formally states that for any $\la\leq \la(N)$  and $T>0$ large
enough there exists a constant $C_T>0$ such that
\begin{equation}\label{observ}
C_T\Big( ||v_1||_{L^2(\Omega)}^{2} + \into \Big[|\n v_0(x)|^2-\la
\frac{v_0^2(x)}{|x|^2}\Big] dx\Big)\leq
\int_{0}^{T}\int_{\Gamma_0}(x\cdot \nu)\Big(\frac{\p v}{\p
\nu}\Big)^2d\sigma dt,
\end{equation}
holds true for $v$ solution of \eqref{eq2}. The main tool to prove
\eqref{observ} relies on the multiplier method and
compactness-uniqueness argument \cite{lions1}. In view of that,
Pohozaev identity provides a direct tool to show that
 the solution of system \eqref{eq2} satisfies
 the multiplier
 identity which formally is given by
\begin{equation}\label{multiident}
\frac{1}{2}\int_{0}^{T}\int_{\Gamma} (x\cdot \nu)\Big(\frac{\p v
}{\p \nu}\Big)^2d\sigma dt =\frac{T}{2}
(||v_1||_{L^2(|\Omega)}^{2}+||v_0||_{H_\la}^2) +\int_{\Omega} v_t
\big(x\cdot \n v+\frac{N-1}{2}v\big)\Big|_{0}^{T}dx,
\end{equation}
producing a ``Hidden regularity" efect for the normal derivative. We
refer to Theorem \ref{Wmult} for a rigorous statement. As a
consequence, the solution of system \eqref{eq2} verifies the reverse
Observability inequality. Then identity \eqref{multiident} together
with   the sharp-Hardy inequality stated in Theorem \ref{tu8}   lead
to \textsl{Observability inequality} \eqref{observ}, fact emphasized
in Theorem \ref{t1}.

\begin{Th}\label{tu8}
Assume $\Omega$  satisfies one of the cases C1-C2 . Then, there
exists a constant $C=C(\Omega)\in \rr$ such that
\begin{equation}\label{equu41}
\int_{\Omega}|x|^2|\nabla v|^2dx\leq
R_{\Omega}^{2}\Big[\int_{\Omega}|\nabla
v|^2dx-\frac{N^2}{4}\int_{\Omega}\frac{v^2}{|x|^2}dx\Big]+C\int_{\Omega}v^2dx\quad
\forall v\in C_{0}^{\infty}(\Omega).
\end{equation}
\end{Th}
The proof of Theorem \ref{tu8} is given in the Appendix.
\begin{obs}
The result of Theorem \ref{tu8}, and precisely the constant
$R_\Omega^2$ which appears in inequality \eqref{equu41}, helps to
obtain the control time $T>T_0=2 R_\Omega$ in \eqref{observ}, which
is sharp from the Geometric Control Condition considerations, see
\cite{BardosLeRa}.
\end{obs}
Although Theorem \ref{tu8} is sharp for our applications to
controllability, it is worth mentioning  that we are able to obtain
a more general result as follows.
\begin{Th}\label{tuu8}
Assume $\Omega$  satisfies one of the cases C1-C2. Let be $\eps>0$.
Then, there exists a constant $C_\eps=C(\Omega, \eps)\in \rr$ such
that
\begin{equation}\label{equuu41}
\int_{\Omega}|x|^\eps|\nabla v|^2dx\leq
R_{\Omega}^{\eps}\Big[\int_{\Omega}|\nabla
v|^2dx-\frac{N^2}{4}\int_{\Omega}\frac{v^2}{|x|^2}dx\Big]+C_\eps\int_{\Omega}v^2dx\quad
\forall v\in C_{0}^{\infty}(\Omega).
\end{equation}
\end{Th}
The proof of Theorem \ref{tuu8} is omitted since it applies the same
steps in the proof of Theorem \ref{tu8}.\\

Finally in  Section \ref{6sec} we will consider the
Schr\"{o}dinger-like process
\begin{equation}\label{eq124S}\left\{\begin{array}{ll}
  iu_{t}-\Delta u-\lambda \frac{u}{|x|^2}=0, & (t,x)\in (0, T)\times\Omega, \\
  u(t,x)=h(t,x), & (t,x)\in (0,T)\times \Gamma_0, \\
  u(t,x)=0, & (t,x)\in (0,T)\times (\Gamma\setminus \Gamma_0), \\
  u(0,x)=u_0(x), &  x\in \Omega, \\
\end{array}\right.
\end{equation}
where the singularity is located on the boundary, and we briefly
discuss the well-posedness and controllability properties. In
Section \ref{4sec} we treat with some open related problems.

The main results of this paper have been announced in a short
presentation in \cite{cristica2}.

\section{Pohozaev identity for $A_\la$}\label{secpoho}

In this Section we rigorously justify the Pohozaev-type identity
associated to $A_\la$. We will discuss in a detail manner the case
C1. The details of the case C2 are let to the reader. In the latter
case we only state the corresponding functional framework, see
Subsection \ref{2sec}.
\subsection{The case C1}\label{1subsec}
 Firstly, we introduce the functional
framework which is used throughout the paper and we discuss some of
its properties.

Assume $\Omega\subset \rr^N$, $N\geq 1$ is a smooth domain which
satisfies the case C1
and fix $\la \leq \la(N)$.
 Thanks to inequality \eqref{oeq3},   we
consider the Hardy functional
\begin{equation}\label{normB}
B_{\la}[u]=\into \Big[|\n u|^2 -\la \frac{u^2}{|x|^2}\Big]dx,
\end{equation}
which is positive and finite for all $u\in C_{0}^{\infty}(\Omega)$.
For any $ \la \leq \la(N)$, $B_\la[u]$ induces  a Hilbert space
$H_\la$, defined by the completion of $C_0^\infty(\Omega)$ functions
in the norm
\begin{equation}\label{defnorm}
||u||_{H_\lambda}^{2}=B_\la[u],  \q u\in C_{0}^{\infty}(\Omega).
\end{equation}
 We point out that the space $H_\la$ was firstly introduced by
Vazquez and Zuazua \cite{vazzua} in the case of interior
singularity. As emphasize above, it may be extended to the case of
boundary singularity.  In the subcritical case $\lambda<\lambda(N)$,
it holds that
 $\hoi=H_\la$, according to the estimates
$$ \big(1-\lambda/\la(N)\big)
||u||_{\hoi} \leq ||u||_{H_\lambda}^{2}\leq ||u||_{\hoi}^{2}, \q
\forall \q u\in C_{0}^{\infty}(\Omega),$$ which ensure the
equivalence of the norms.

The critical space $H_{\la(N)}$  turns to be slightly larger than
$\hoi$. Remark that $B_{\la(N)}[u]$ is finite for any $u\in \hoi$,
but it makes sense as an improper integral approaching the singular
pole $x=0$ (see the right hand side of \eqref{density})
for more general distributions. 
 As happens in the case of interior singularity
 (see \cite{hiddenenergy}), in general the meaning of
$||u||_{H_{\la(N)}}$ does not coincide  with the improper integral
of $B_{\la(N)}[u]$. Following $\cite{hiddenenergy}$, we can
construct a counterexample even in the case when the singularity is
located on the border.
 Indeed,
we fix $\Omega:= \{  x\in \rr_{+}^{N} \ : \ |x'|^2+(x_N-1)^2\leq
1\}$ and we consider the distribution
$e_1=x_N|x|^{-N/2}J(z_{0,1}|x|)$ where $z_{0,1}$ is the first
positive zero of the Bessel function $J_0$. We observe that
$B_{\la(N)}[u]$ is finite
 as an improper integral approaching the origin.
  On the other hand,  computing we  remark that
$$||e_1-\phi||_{H_{\la(N)}}\geq C_0>0, \q \forall \phi \in C_{0}^{\infty}(\Omega),$$
for some positive universal  constant $C_0>0$.  This is in
contradiction with the definition of $H_{\la(N)}$ which allows the
existence of a sequence $\phi_n\in C_{0}^{\infty}(\Omega)$
converging to $e_1$ in $H_{\la(N)}$-norm ! Thererefore, the
assumption of considering the definition of the $H_{\la(N)}$-norm as
an improper integral of $B_{\la(N)}$  is false. In other words,
there are distributions $u\in H_{\la(N)}$ for which
\begin{equation}\label{density}
||u||_{H_{\la(N)}}^2\neq \lim_{\eps\rightarrow 0}\int_{|x|\geq
\eps}\Big[ |\n u|^2-\la(N)\frac{u^2}{|x|^2}\Big]dx.
\end{equation}
Next we propose an equivalent norm on $H_\la$, $\la\leq \la(N)$,
which overcomes the anomalous behavior in \eqref{density} and
describes perfectly the meaning of the  $H_\la$-norm.

\subsubsection{The meaning of $H_\la$-norm}
For reasonable considerations that will be precise in
\eqref{identity}, we introduce the functional
\begin{equation}\label{norm2SB}
B_{\la, 1}[u]=\int_\Omega \Big|\n u+ \frac{N}{2}
\frac{x}{|x|^2}u-\frac{e_N}{x_N}u\Big|^2dx+ (\la(N)-\la)\into
\frac{u^2}{|x|^2}dx.
\end{equation}
which is positive and finite for any $u\in C_{0}^{\infty}(\Omega)$
and $\la\leq \la(N)$. Here, we have denoted by $e_N$ the $N-$th
canonical vector of $\rr^N$.  Next, we observe that, for any
$\la\leq \la(N),$
\begin{equation}\label{identity}
B_\la[u]=B_{\la, 1}[u], \q \forall u\in C_{0}^{\infty}(\Omega).
\end{equation}
Besides, notice that both $B_{\la, 1}[u]$ and $B_\la[u]$ are norms
in $H_\la$ and they coincide on $C_{0}^{\infty}(\Omega)$. Due to
definition \eqref{defnorm} of $H_\la$, we conclude that
 the $H_\la$ could be define as the closure of
 $C_{0}^{\infty}(\Omega)$ in the norm
  induced by $B_{\la, 1}[u]$. Therefore, the $H_\la$-norm is characterized by the identification
\begin{equation}\label{newnormB}
||u||_{H_\la}^2=\lim_{\eps\rightarrow 0}B_{\la, 1}^{\eps}[u], \q
\forall u\in H_{\la}, \end{equation}
 where $\la\leq \la(N)$ and
\begin{equation*}
B_{\la, 1}^{\eps}[u]:=\int_{|x|\geq \eps} \Big|\n u+ \frac{N}{2}
\frac{x}{|x|^2}u-\frac{e_N}{x_N}u\Big|^2dx+
(\la(N)-\la)\int_{|x|\geq \eps} \frac{u^2}{|x|^2}dx, \q \forall u\in
H_{\la}.\end{equation*}

Next in the paper we will understand the meaning of the norm
$||\cdot||_{H_\la}$ as in formula \eqref{newnormB}.

\subsubsection{Main
results}\label{3sec} First of all, we note that standard elliptic
estimates do not apply for $A_\la $ to obtain enough regularity for
the normal derivative since the singularity $x=0$ is located on the
boundary.
 However, the following trace regularity result stated in Theorem \ref{trace}
 holds true. In what follows, $D(A_\la)$ stands for the domain of $A_\la$ defined
in \eqref{domeniu}.

Next,  we claim the main results of Section \ref{secpoho}.

\begin{Th}[Trace regularity]\label{trace}
Assume $\Omega\subset\rr^N$, $N\geq 1$, is a bounded smooth domain
satisfying the case C1.  Let us consider  $\la \leq \la(N)$ and
$u\in D(A_\la)$. Then
\begin{equation}\label{tracereg}
\Big(\frac{\p u }{\p \nu}\Big)^2|x|^2\in L^1(\Gamma),
\end{equation}
and moreover, there exists a positive constant $C=C(\Omega)>0$ such
that
\begin{equation}\label{ineqtrace}
\int_{\Gamma}\Big(\frac{\p u }{\p \nu}\Big)^2|x|^2 d\sigma \leq
C(||u||_{H_\la}^2+||A_\la u||_{L^2(\Omega)}^2), \q \forall \q u\in
D(A_\la).
\end{equation}
\end{Th}

Moreover, we obtain the following

\begin{Th}[Pohozaev identity]\label{poho}
Assume $\Omega\subset \rr^N$, $N\geq 1$, is a smooth bounded domain
satisfying the case C1 and let $\la\leq \la(N)$. 
If $u\in D(A_\la)$ we  claim that
\begin{align}\label{pohoidentity}
\frac{1}{2}\int_{\Gamma}(x\cdot \nu) \Big(\frac{\p u}{\p \nu}\Big)^2
d\sigma&= -\into A_\la u(x\cdot \n u)dx-
\frac{N-2}{2}||u||_{H_\la}^2,
\end{align}
\end{Th}

The proofs of Theorems \ref{trace}, \ref{poho} are quite technical,
so we need to apply some preliminary lemmas which are  stated below.
The proofs of  Lemmas  \ref{lem}, \ref{unifbound} are postponed at
the end of Subsection \ref{1subsec} while Lemma \ref{lemalimit} is a
consequence of an abstract approximation lemma in a forthcoming work
\cite{acz}.

\begin{lema}\label{lem}
Supppose $u\in D(A_\la)$ and  denote $f=A_\la u\in L^2(\Omega)$. Let
us also consider  $\theta_\eps \in C_{0}^{\infty}(\Omega)$,
$\eps>0$, a family of cut-off functions such that
\begin{equation}\label{theaaprox}
\theta_\eps(x)=\theta_\eps(|x|)=\left\{\begin{array}{ll}
  0, & |x|\leq \eps \\
  1, & |x|\geq 2\eps. \\
\end{array}\right.
\end{equation}
Assume  $\vec{q}\in (C^2(\overline{\Omega}))^N$ is  a vector field
such that $\vec{q}=\nu$ on $\Gamma$, where $\nu$ denotes the outward
normal  to the boundary $\Gamma$ (such an election of $\vec{q}$ can
be always done in smooth domains, see \cite{lions1}, Lemma 3.1, page
29). Then we have the identity
\begin{align}\label{lastform}
\frac{1}{2}\int_{\Gamma} \Big(\frac{\p u}{\p
\nu}\Big)^2|x|^2\theta_\eps d\sigma&=-\into f(|x|^2 \vec{q}\cdot \n
u \theta_\eps)dx +2\into (x\cdot \n
u)(\vec{q}\cdot \n u)\theta_\eps dx\nonumber\\
&+\sum_{i, j=1}^{N} \into u_{x_i}u_{x_j} |x|^2q_{x_i}^j \theta_\eps
dx-\into |\n u|^2 (x\cdot \vec{q})\theta_\eps dx\nonumber\\
&-\frac{1}{2}\into \textrm{div}\vec{q}|x|^2\Big[|\n u|^2-\la
\frac{u^2}{|x|^2}\Big]\theta_\eps dx -\frac{1}{2} \into |x|^2
\vec{q} \cdot \n \theta_\eps \Big[|\n u|^2-\la
\frac{u^2}{|x|^2}\Big]
dx\nonumber\\
&+\into |x|^2(\vec{q}\cdot \n u)(\n u\cdot \n \theta_\eps) dx.
\end{align}
\end{lema}

\begin{lema}\label{lemalimit}
Assume $f\in L^2(\Omega)$ and $\Omega\subset \rr^N$ verifying  the
case C1. 
 Let $\eps>0$ be small enough.  We
consider the following approximation problem
\begin{equation}\label{rel2}
\left \{\begin{array}{ll}
  A_{\la(N)-\eps}u_\eps=f,  & x\in \Omega \\
  u_\eps=0, & x\in \p \Omega.
\end{array}\right.
\end{equation}
Then $$u_\eps \rightarrow u \q \textrm{ strongly in } H_{\la(N)}, \q
\textrm{ as } \eps \rightarrow 0.$$ where $u$ verifies the limit
problem
       $$-\D u -\la(N) \frac{u}{|x|^2}=f, \textrm{ in } \mathcal{D}'(\Omega).$$
       Moreover, \begin{equation}
       \eps \into \frac{u_{\eps}^2}{|x|^2}dx \rightarrow 0, \textrm{ as
       } \eps \rightarrow 0.
       \end{equation}
\end{lema}
\begin{lema}\label{unifbound}
Assume $\Omega$ fulfills the case C1 and let $\la\leq \la(N)$.  Let
$f\in C^\infty(\Omega)$. Moreover, we assume that $u_\la$ solves the
problem
\begin{equation}\label{approxB}
\left\{
\begin{array}{ll}
  A_{\la}u_\la=f, & x\in \Omega,  \\
  u_\la\in H_\la.  &  \\
\end{array}
\right.
\end{equation}
Then  $u_\la$ satisfies the following upper bounds: there exists
$r_0< R_\Omega$ small enough and there exist constants $C_1, C_2>0$,
independent of $\la$, such that
\begin{equation}\label{uppbound1}
|u_\la(x)|\leq C_1 x_N |x|^{-N/2+\sqrt{\la(N)-\la} }\Big|\log
\frac{1}{|x|}\Big|^{1/2}, \q \textrm{ a.e.  } x\in \Omega_{r_0},
\end{equation}
\begin{equation}\label{uppbound2}
|\n u_\la(x)|\leq C_2  |x|^{-N/2+\sqrt{\la(N)-\la}} \Big|\log
\frac{1}{|x|}\Big|^{1/2}, \q \textrm{ a.e. } x\in \Omega_{r_0},
\end{equation}
where $\Omega_{r_0}:=\Omega\cap B_{r_0}(0)$.
\end{lema}

\noindent {\bf Notation:} In order to facilitate the computations,
in the sequel, we will write $``\gtrsim "$ and $`` \lesssim "$
instead of $``\geq C"$ respectively $``\leq C"$ when we refer to
universal constants $C$.

\subsubsection{Proofs of Theorems \ref{trace}, \ref{poho}}
\vspace{.2cm}

\begin{proof}[Proof of Theorem \ref{trace}]

Following the  proof of Theorem \ref{tu8}, we are able to show that,
\begin{equation}
\into |x| |\n u|^2 dx\lesssim ||u||_{H_\la}^2, \q \forall \ u\in
H_\la.
\end{equation}
From the above estimate and Cauchy-Schwartz inequality applied to
identity \eqref{lastform} in Lemma \ref{lem} we reach to
\begin{equation}\label{ineqtraceeps}
\int_{\Gamma}\Big(\frac{\p u }{\p \nu}\Big)^2|x|^2 \theta_\eps
d\sigma \lesssim ||u||_{H_\la}^2+||f||_{L^2(\Omega)}^2, \q \forall
\q u\in D(A_\la), \q \forall \q  \eps>0.
\end{equation}
Thanks to Fatou Lemma we finish the proof of Theorem \ref{trace}.

\end{proof}

\begin{proof}[Proof of Theorem \ref{poho}]

We split the proof in two main steps.\\

\noindent \textsl{ Step 1. The subcritical case}\\

 Note that
$H_\la=\hoi$. Let $u\in D(A_\la)$ and put
$f:=A_\la u \in L^2(\Omega).$
By standard elliptic estimates  we note that $u\in
H^2(\Omega\setminus B_\eps(0))$, for any $\eps>0$ small enough.
Moreover, the normal derivative $\p u/\p \nu\in
L_{loc}^2(\p\Omega\setminus \{0\})$. We multiply $A_\la u$ by
$x\cdot \n u\theta _\eps$,  where $\theta_\eps$, $\eps>0,$ was
defined in \eqref{theaaprox}. After integration we get
\begin{align}\label{poho1}
\frac{1}{2}\int_{\Gamma} (x\cdot \nu)\Big(\frac{\p u}{\p \nu}\Big)^2
\theta_\eps d\sigma&= -\into f(x\cdot \n u) \theta_\eps dx-
\frac{N-2}{2}\into \Big[|\n u|^2-\la
\frac{u^2}{|x|^2}\Big]\theta_\eps dx\nonumber\\
&-\frac{1}{2}\into \Big[|\n u|^2-\la \frac{u^2}{|x|^2}\Big]x\cdot \n
\theta_\eps dx +\into (x\cdot \n u)(\n u \cdot \n \theta_\eps)dx.
\end{align}
 Combining the Dominated Convergence Theorem (DCT) with Theorem \ref{trace},
 the left hand side of \eqref{poho1} converges i.e.
$$\int_{\Gamma}(x\cdot \nu) \Big(\frac{\p u}{\p \nu}\Big)^2\theta_\eps d \sigma
\rightarrow \int_{\Gamma} (x\cdot \nu) \Big(\frac{\p u}{\p
\nu}\Big)^2 d \sigma, \q \textrm{ as }\eps\rightarrow 0.$$ In the
right hand side, we can directly pass to the limit term by term to
obtain the identity \eqref{pohoidentity} as follows. Firstly, since
$x\cdot \n u\in L^2(\Omega)$ we have that
\begin{align*}\left\{
\begin{array}{ll}
&|f(x\cdot \n u)\theta_\eps|\leq |f||x\cdot \n u|\in L^1(\Omega),\\
&\theta_\eps \rightarrow 1, \textrm { a.e. as} \eps \rightarrow 0,
\end{array}\right.
\end{align*}
and by DCT we obtain
$$\into f(x\cdot
\n u) \theta_\eps dx\rightarrow
\into f(x\cdot \n u)dx, \q \textrm{ as } \eps \rightarrow 0.$$ 
Besides, from Hardy inequality and DCT we have
$$\into |\n u|^2 \theta_\eps dx \rightarrow \into |\n u|^2 dx, \q \into \frac{u^2}{|x|^2}\theta_\eps dx\rightarrow
\into \frac{u^2}{|x|^2}dx, \q \textrm{ as } \eps \rightarrow 0.$$
Using the fact that $|\n \theta_\eps|= O(1/\eps)$ it follows that
$$\Big|\into |\n u|^2 x\cdot \n \theta_
\eps dx\Big|\lesssim \int_{B_{2\eps}\setminus B_\eps}|\n u|^2
dx\rightarrow 0,$$
$$\Big|\into \frac{u^2}{|x|^2} x\cdot \n \theta_\eps dx\Big|\lesssim \int_{B_{2\eps}\setminus B_\eps}\frac{u^2}{|x|^2} dx\rightarrow
0,$$
$$\Big|\into (x\cdot \n u)(\n u\cdot \n \theta_\eps)
dx\Big|\lesssim \int_{B_{2\eps}\setminus B_\eps}|\n u|^2
dx\rightarrow 0,
$$
as $\eps \rightarrow 0$. With these we conclude the solvability of
Theorem \ref{poho} in the
subcritical case $\la< \la(N)$.\\

\noindent \textsl{Step 2. The critical case $\la=\la(N)$}\\

 As before, let
us consider $u\in D(A_{\la(N)})$ and $f:=A_{\la(N)}u\in
L^2(\Omega)$. Our purpose is to show the validity of Theorem
\ref{poho} for such $u$.

We proceed by approximations with subcritical values. More
precisely, for $\eps>0$ small enough we consider the problem
\begin{equation}\label{approx}
\left\{
\begin{array}{ll}
  A_{\la(N)-\eps}u_\eps=f, & x\in \Omega,  \\
  u_\eps\in \hoi.  &  \\
\end{array}
\right.
\end{equation}
Applying  Lemma \ref{lemalimit} we obtain
\begin{equation}\label{limit}
u_\eps \rightarrow u \textrm{ in } H_{\la(N)}, \q  \eps \into
\frac{u_\eps^2}{|x|^2}dx\q \textrm{ as }   \eps\rightarrow 0,
\end{equation}%
where $u$ solves the limit problem. According to the Pohozaev
identity applied to $u_\eps$  we reach to
\begin{align}\label{pohoidentitysub}
\frac{1}{2}\int_{\Gamma} (x\cdot \nu) \Big(\frac{\p u_\eps}{\p
\nu}\Big)^2 d\sigma&= -\into f(x\cdot \n u_\eps)dx-
\frac{N-2}{2}\Big(||u_\eps||_{H_{\la(N)}}^2+\eps\into
\frac{u_\eps^2}{|x|^2}dx\Big).
\end{align}
Due to Theorem \ref{tu8}, the fact that $u_\eps \rightarrow u$ in
$H_{\la(N)}$ implies
  $$x\cdot \n u_\eps \rightarrow x\cdot \n u \ \textrm{ in } L^2(\Omega),
   \ \textrm{ as }
  \eps \rightarrow 0. $$
Therefore,  the right hand side in \eqref{pohoidentitysub} converges
to
$$-\into f(x\cdot \n u)dx -\frac{N-2}{2}||u||_{H_{\la(N)}}^2:=H(u),$$
and therefore
$$\textrm{ there exists } \lim_{\eps \rightarrow 0} \frac{1}{2}\int_{\Gamma}
(x\cdot \nu)\Big(\frac{\p u_\eps}{\p \nu}\Big)^2 d\sigma=H(u).$$ On
the other hand, by standard elliptic regularity one can show that
$$\frac{\p u_\eps}{\p \nu}\rightarrow  \frac{\p
u}{\p \nu}  \textrm{ in } L_{\textrm{loc}}^2(\Gamma\setminus \{0\})
\textrm{ and } \frac{\p u_\eps}{\p \nu}\rightarrow  \frac{\p u}{\p
\nu} \textrm{ a.e.  on }\Gamma.$$ In the sequel, we
discuss two different situations for the geometry of $\Omega$.\\

 \noindent {\bf
Case 1.} Assume $\Omega$ is flat in a neighborhood of zero (i.e.
$x\cdot \nu=0$). Then it easily to note that
$$\lim_{\eps\rightarrow 0}\int_{\Gamma}
(x\cdot \nu)\Big(\frac{\p u_\eps}{\p \nu}\Big)^2
d\sigma=\int_{\Gamma} (x\cdot \nu)\Big(\frac{\p u}{\p \nu}\Big)^2
d\sigma.$$ In consequence, $u$ satisfies the Pohozaev identity, by
passing to the limit in
\eqref{pohoidentitysub}.\\

\noindent  {\bf Case 2.}  We assume  $\Omega$ is not necessary  flat
at
origin. We distinguish two cases when discussing the smoothness of $f$.\\ 

\noindent \textit{The case $f\in C^\infty(\Omega)$.}\\
%
Next we apply Lemma \ref{unifbound} for $u_\eps$ the solution of
problem \eqref{approx}.
 and we obtain
$$\Big|(x\cdot \nu)\Big(\frac{\p u_\eps}{\p \nu}\Big)^2\Big|\leq
\Big(\frac{\p u_\eps}{\p \nu}\Big)^2|x|^2\leq g, \textrm{ a.e. on
}\Gamma,$$ where $g=|x|^{2-N}\Big|\log\frac{1}{|x|}\Big|\in
L^1(\Gamma)$. Applying DCT we conclude

$$\lim_{\eps\rightarrow 0}\int_{\Gamma}
(x\cdot \nu)\Big(\frac{\p u_\eps}{\p \nu}\Big)^2
d\sigma=\int_{\Gamma} (x\cdot \nu)\Big(\frac{\p u}{\p \nu}\Big)^2
d\sigma.$$

\noindent \textit{The case $f\in L^2(\Omega)$.}\\

 We consider
$\{f_k\}_{k\geq 1}\in C^\infty(\Omega)$ such that $f_k\rightarrow f$
in $L^2(\Omega)$, as $k\rightarrow \infty$.

Let us call $u_k$ the solution of $A_{\l(N)}u_k=f_k$, for all $k\geq
1$. From the previous case,   $u_k$ satisfies
\begin{align}\label{pohoidentityB}
\frac{1}{2}\int_{\Gamma} (x\cdot \nu)\Big(\frac{\p u_k}{\p
\nu}\Big)^2 d\sigma&= -\into f_k (x\cdot \n u_k)dx-
\frac{N-2}{2}||u_k||_{H_\la}^2.
\end{align}

 We know that $f_k$ is a Cauchy sequence in $L^2(\Omega)$, and
due to
$$||u_k-u_l||_{H_\la(N)}\lesssim  ||f_k-f_l||_{L^2(\Omega)}
\rightarrow 0, \textrm{ as } k, l\rightarrow \infty,$$ we deduce
that $\{u_k\}_{k\geq 1}$ is Cauchy in $H_{\la(N)}$. Hence $u_k
\rightarrow u$ in $H_{\la(N)}$ and
$$x\cdot \n u_k \rightarrow x\cdot \n u \textrm{ in } L^2(\Omega).$$
As a consequence we can pass to the limit in the right hand side of
\eqref{pohoidentityB}. In order to finish the proof, next we also
show we can also pass to the limit in the left hand side. Indeed, in
view of  Theorem \ref{trace} we have
$$\int_{\Gamma} \Big(\frac{\p (u_k-u_l)}{ \p \nu} \Big)^2 |x|^2
d\sigma \lesssim ||u_k-u_l||_{H_\la}^2+||f_k-f_l||_{L^2(\Omega)}.
$$
Therefore $g_k:=\frac{\p u_k}{\p \nu} |x|$ is a Cauchy sequence in
$L^2(\Gamma)$ and $g_k\rightarrow g:=\frac{\p u}{\p \nu} |x|$ in
$L^2(\Gamma)$, as $k$ goes to infinity. This suffices to say that
$$\lim_{k\rightarrow \infty}\int_{\Gamma}(x\cdot \nu)\Big(\frac{\p u_k}{\p \nu}\Big)^2
 d\sigma =\int_{\Gamma}(x\cdot \nu)\Big(\frac{\p u}{\p
\nu}\Big)^2 d\sigma.$$ Therefore we conclude the proof of Theorem
\ref{poho}.
\end{proof}

\subsubsection{Proofs of useful lemmas}
\vspace{0.2cm}
\begin{proof}[Proof of Lemma \ref{lem}]
By standard elliptic estimates, we remark that $u\in
H_{\textrm{loc}}^2(\Omega\setminus\{0\})$. Thanks to that, after
multiplying $f$ by $|x| \vec{q}\cdot \n u
  \theta_\eps$  we are allowed to integrate by parts on $\Omega$.
Firstly, we obtain
\begin{align*}
\into \D u (|x|^2\vec{q}\cdot \n u
\theta_\eps)dx&=\int_{\Gamma}\frac{\p u}{\p \nu} (|x|^2\vec{q}\cdot
\n u \theta_\eps)d\sigma-\into \n u \cdot \n (|x|^2\vec{q}\cdot \n u
\theta_\eps)dx.
\end{align*}
Let us now compute the boundary term above. Since $u$ vanishes on
$\Gamma$ it follows that
\begin{equation} \n u=\frac{\p u}{\p
\nu}\nu, \q \textrm{ on } \Gamma, \end{equation}
 and moreover, $\vec{q}=\nu$ on $\Gamma$.
Thanks to these we obtain
$$\int_{\Gamma} \frac{\p u}{\p \nu} (|x|^2\vec{q}\cdot \n
u \theta_\eps)d\sigma=\int_{\Gamma} \Big(\frac{\p u}{\p
\nu}\Big)^2|x|^2\theta_\eps d\sigma.$$ Therefore,
\begin{align*}
\into \D u (|x|^2\vec{q}\cdot \n u \theta_\eps)dx
&=\int_{\Gamma} \Big(\frac{\p u}{\p \nu}\Big)^2|x|^2\theta_\eps
d\sigma- \into \n u\cdot \n (|x|^2\vec{q}\cdot \n
u)\theta_\eps\nonumber\\
&-\into |x|^2 (\vec{q}\cdot \n u)(\n u\cdot \n \theta_\eps) dx.
\end{align*}
Let us compute the second term in the integration above. Doing
various iterations we obtain
\begin{align}\label{secondterm}
\into \n u\cdot \n (|x|^2\vec{q}\cdot \n u)\theta_\eps&=
2\into (x\cdot \n u)(\vec{q}\cdot \n u)\theta_\eps dx +\sum_{i,
j=1}^{N} \into u_{x_i}u_{x_j} |x|^2q_{x_i}^j \theta_\eps dx\nonumber\\
&+ \frac{1}{2}\sum_{i, j=1}^{N} \into|x|^2q^j (u_{x_i}^{2})_{x_j}
\theta_\eps d\sigma
\end{align}
For the last term in the  integration above we get
\begin{align}\label{bound}
\frac{1}{2}\sum_{i, j=1}^{N} \into |x|^2q^j (u_{x_i}^{2})_{x_j}
\theta_\eps d\sigma
&= \frac{1}{2}\int_{\Gamma} \Big(\frac{\p u}{\p
\nu}\Big)^2|x|^2\theta_\eps d\sigma -\into |\n u|^2
(x\cdot \vec{q})\theta_\eps dx\nonumber\\
&-\frac{1}{2}\into \textrm{div}\vec{q}|x|^2|\n u|^2 \theta_\eps dx
-\frac{1}{2} \into |x|^2|\n u|^2 \vec{q} \cdot \n \theta_\eps dx.
\end{align}
According to  \eqref{secondterm} and \eqref{bound} we obtain
\begin{align}\label{formLaplace}
\into \D u (|x|^2\vec{q}\cdot \n u
\theta_\eps)dx&=\frac{1}{2}\int_{\Gamma} \Big(\frac{\p u}{\p
\nu}\Big)^2|x|^2\theta_\eps d\sigma-2\into (x\cdot \n
u)(\vec{q}\cdot \n u)\theta_\eps dx\nonumber\\\
&-\sum_{i, j=1}^{N} \into u_{x_i}u_{x_j} |x|q_{x_i}^j \theta_\eps
dx+\into |\n u|^2
(x\cdot \vec{q})\theta_\eps dx\nonumber\\
&+\frac{1}{2}\into \textrm{div}\vec{q}|x|^2|\n u|^2 \theta_\eps dx
+\frac{1}{2} \into |x|^2|\n u|^2 \vec{q} \cdot \n \theta_\eps
dx\nonumber\\
&-\into |x|^2(\vec{q}\cdot \n u)(\n u\cdot \n \theta_\eps) dx.
\end{align}
On the other hand, it follows that
\begin{align}\label{formsingular}
\into \frac{u}{|x|^2}|x|^2\vec{q}\cdot \n u \theta_\eps dx
&= -\frac{1}{2} \into \textrm{div} \vec{q} u^2 \theta_\eps
dx-\frac{1}{2} \into \vec{q} \cdot \n \theta_\eps u^2 dx.
\end{align}
From  \eqref{formLaplace} and \eqref{formsingular} we finally obtain
the identity of Lemma \ref{lem}.
\end{proof}

\begin{proof}[Proof of Lemma \ref{unifbound}]
For any $\la\leq \la(N)$ we fix $\phi_\la=x_N
|x|^{-N/2+\sqrt{\la(N)-\la} }\Big|\log \frac{1}{|x|}\Big|^{1/2}$.
Let us also consider the problem
\begin{equation}\label{approx2}
\left\{
\begin{array}{ll}
  A_{\la}U_\la=|f|, & x\in \Omega,  \\
  U_\la\in H_\la.  &  \\
\end{array}
\right.
\end{equation}
The proof follows several steps.\\

\noindent \textit{{Step 1}}.
 Firstly let us check the validity
of the Maximum Principle:
 \begin{equation}\label{MP}
 |u_\la(x)|\leq U_{\la}(x)\q \textrm{a.e. in }\Omega.
 \end{equation}
 Indeed, from the equations satisfied by $U_\la$, $u_\la$ we obtain
\begin{equation}\label{difference}
-\D (U_\la\pm u_\la)-\la \frac{(U_\la\pm u_\la)}{|x|^2}=|f|\pm f\geq
0, \q \forall \ x\in \Omega.
\end{equation}
Multiplying \eqref{difference} by the negative part $(U_\la\pm
u_\la)^{-}$ we get the reverse Hardy inequality
\begin{equation}\label{contrhardy}
\into \Big[ |\n (U_\la\pm u_\la)^{-}|^2-\la \frac{[(U_\la\pm
u_\la)^{-}]^2}{|x|^2} \Big] dx \leq 0.
\end{equation}
From the non-attainability of the Hardy constant we necessary must
have $(U_\la\pm u_\la)^{-}\equiv 0$ in $\Omega$. Therefore,
$U_\la\pm u_\la\geq 0$ in $\Omega$, fact which concludes
\eqref{MP}.\\

\noindent \textit{{Step 2}}.
 Next, we remark that there exists a positive constant
$C>0$, independent of $\la$ such that
$$-\D \phi_\la-\la\frac{\phi_\la}{|x|^2}\geq C_1, \q \forall x\in \Omega.$$
Therefore, for $C\geq ||f||_{L^\infty}/{C_1}$ we get
\begin{equation}\label{comparison}
\left\{\begin{array}{ll}
  -\D(C\phi_\la-U_\la)-\la\frac{(C\phi_\la-U_\la)}{|x|^2}\geq 0, &
  \forall x\in\Omega,  \\
  C\phi_\la-U_\la\geq 0, & x\in \Gamma. \\
\end{array}\right.
\end{equation}
Therefore, applying the Maximum Principle we obtain
\begin{equation}\label{upperbound}
U_\la\leq C \phi_\la, \q \forall x\in \Omega,\q  \la\leq \la(N),
\end{equation}
 and the proof \eqref{uppbound1} is finished.\\

\noindent \textit{{Step 3}}.  For the estimate \eqref{uppbound2} we
use a remark by Brezis-Marcus-Shafrir \cite{marcus2} as follows.

Fix $x\in \Omega_{r_0/2}$ and put $r=|x|/2$. We define then
$\tilde{u}_\la(y)=u_\la(x+ry)$ where $y\in B_1(0)$. By direct
computations we obtain
\begin{align}\label{relation}
\D {\tilde{u}_\la(y)}&=r^2 \D u_\la(x+r y)= r^2 \Big(-f -\la
\frac{u_\la(x+ry)}{|x+ry|^2}\Big)\nonumber\\
&= -r^2f-\la\frac{|x|^2}{4|x+ry|^2}\tilde{u}_\la(y).
\end{align}

On the other hand, we remark that $$\frac{4}{9}\leq
\frac{|x|^2}{|x+ry|^2}\leq 4, \q \forall \q y\in B_1(0). $$
 By elliptic estimates it is easy to see
 that $\tilde{u}_\la\in
C^1(B_1(0))$.  Applying the interpolation inequality (see Evans
\cite{MR1625845}),  we get that
\begin{align}\label{interp}
|\n \tilde{u}_\la(0)|&\lesssim
||\tilde{u}_\la||_{L^\infty(B_1(0))}+||\D
\tilde{u}_\la||_{L^\infty(B_1(0))} \nonumber\\
&\lesssim
||\tilde{u}_\la||_{L^\infty(B_1(0))}+||f||_{L^\infty(\Omega)}
\end{align}
Writing $\n \tilde{u}_\la$ in terms of $\n u_\la$ we obtain
\begin{equation}\label{gradbound}
|\n u_\la(x)|\lesssim
\frac{1}{|x|}(||\tilde{u}_\la||_{L^\infty(B_1(0))}+||f||_{L^\infty})
\end{equation}
In addition, from \eqref{upperbound} we have
\begin{align}\label{inft}
||\tilde{u}_\la||_{L^\infty(B_1(0))}&=||u_\la(x+ry)||_{L^\infty(B_1(0))}\nonumber\\
& \lesssim  \sup_{y\in B_1(0)}\Big\{(x_N+ry_N)|x+r
y|^{-N/2+\sqrt{\la(N)-\la}}\Big|\log\frac{1}{|x+ry|}\Big|^{1/2}\Big\}\nonumber\\
&\lesssim  x_N|x|^{-N/2+\sqrt{\la(N)-\la}}\Big|\log
\frac{1}{|x|}\Big|^{1/2}+|x|^{-(N-2)/2+\sqrt{\la(N)-\la}}\Big|\log
\frac{1}{|x|}\Big|^{1/2}\nonumber\\
&\lesssim  |x|^{-(N-2)/2+\sqrt{\la(N)-\la}}\Big|\log
\frac{1}{|x|}\Big|^{1/2} ,
\end{align}
which is verified for all $ x\in \Omega_{r_0}$,  $y\in B_1(0).$ From
\eqref{gradbound} and \eqref{inft} we obtain the estimate
\eqref{uppbound2} which yield the proof of Lemma \ref{unifbound}.
\end{proof}

\subsection{Applications to semi-linear equations}\label{8sec}

Pohozaev-type identities apply mostly  to show non-existence results
 for elliptic problems. In what follow we emphasize a direct
application to a non-linear elliptic equation with boundary singular
potential. To fix the ideas, let us  assume $\la< \la(N)$ and
consider $\Omega\subset \rr^N$, $N\geq 1$,  a  domain  satisfying
the case C1. Next
$$\alpha_{\star}:=\frac{N+2}{N-2}$$ stands for  the critical Sobolev exponent.

Next we claim
\begin{Th}\label{th1} Let us consider the problem
\begin{equation}\label{nonexist}
\left\{ \begin{array}{ll}
  -\D u-\frac{\lambda}{|x|^2}u=|u|^{\alpha-1} u,   & x\in \Omega,  \\
  u=0, & x\in \Gamma. \\
\end{array}\right.
\end{equation}
\begin{enumerate}[1.]
\item\label{a1} Assume $\la\leq \la(N)$.   If $1<\alpha< \alpha_\star$
 the problem \eqref{nonexist} has non trivial solutions in $H_\la$.
 Moreover, if $1<\alpha< \frac{N}{N-2}$  the problem \eqref{nonexist} has
non trivial solutions in $D(A_\la)$.
\item\label{a2} (non-existence). Assume  $\la\leq \la(N)$ and let $\Omega$
be a smooth star-shaped domain (i.e. $x\cdot \nu\geq 0$, for all
$x\in \Gamma$).  If $\alpha \geq \alpha_\star$ the problem does not
have non trivial solutions in $ D(A_{\la})$.
\end{enumerate}

\end{Th}

\noindent {\bf Proof of Theorem \ref{th1}}
\begin{proof}[Proof of \ref{a1}]
 The existence of non trivial solutions for \eqref{nonexist}
reduces to study the minimization problem
$$I=\inf_{u\in H_\la, u\neq 0}
\frac{||u||_{H_\la}^2}{||u||_{L^{\alpha+1}(\Omega)}^{\alpha+1}}.$$
Without losing the generality,  we may consider the normalization
\begin{equation}\label{normaliz}
I=\inf_{||u||_{L^{\alpha+1}(\Omega)}=1} J(u),
\end{equation}
where $J:H_\la\rightarrow \rr$, $J(u)=||u||_{H_\la}^{2}$ and we
address the question of attainability of $I$ in \eqref{normaliz}.\\

\noindent 
 We note that $J$ is
continuous, convex, coercive in $H_\la$. Let $\{u_n\}_n$ be a
minimizing sequence of $I$, i.e.,
$$J(u_n)\searrow I, \q ||u_n||_{L^{\alpha+1}(\Omega)}=1.$$
By the coercivity of $J$ we  have
$$||u_n||_{H_\la}\leq C, \q \forall n,$$
Moreover,  the embedding $H_\la \hookrightarrow
L^{\alpha+1}(\Omega)$ is compact for any $\alpha <\alpha_\star$ (it
can be deduced combining Theorem \ref{tuu8} and  Sobolev
inequality). Therefore,
\begin{equation}\left\{
\begin{array}{ll}
  u_n \rightharpoonup u & \textrm{weakly in } H_\la, \\
  u_n \rightarrow u & \textrm{ strongly in } L^{\alpha+1}(\Omega). \\
\end{array}\right.
\end{equation}
Therefore, $||u||_{L^{\alpha+1}(\Omega)}=1$. From the i.s.c. of the
norm we have
$$I\leq J(u)\leq \liminf_{n\rightarrow \infty}J(u_n)=I,$$ and
therefore $I=J(u)$ is attained by $u$, which, up to a constant, is a
non-trivial solution of \eqref{nonexist} in $H_\la$.

If $\alpha< N/(N-2)$ let us show that $u \in D(A_\la)$. Indeed,  due
to the compact embedding $H_\la\hookrightarrow L^q(\Omega)$,
$q<2N/(N-2)$, we have that $|u|^{\alpha-1}u\in L^2(\Omega)$. In
consequence, $u\in D(A_\la)$.
\end{proof}
\begin{proof}[Proof of \ref{a2}]

For the proof of the non-existence part we apply  the Pohozaev
identity in Theorem \ref{poho}. In view of that we use the following
lemma whose proof is postponed at the end of the section.
\begin{lema}\label{lem2}
Assume $\la\leq \la(N)$ and $1< \alpha< \infty$. Then, any solution
$u\in  D(A_\la)$ of \eqref{nonexist} satisfies the identity
\begin{equation}\label{negative}
\frac{1}{2}\int_{\Gamma}(x\cdot \nu)\Big(\frac{\p u}{\p \nu}\Big)^2
d \sigma =\Big(\frac{N}{1+\alpha}-\frac{N-2}{2}\Big)\into
|u|^{\alpha+1}dx.
\end{equation}
\end{lema}

\noindent  \textsl{The case $\alpha> \alpha_\star$.}\\

 Note  that $x\cdot \nu \geq 0$
  for all $x\in \Gamma$.  Assuming
$u\not \equiv 0$, from Lemma \ref{lem2} we obtain $(N-2)/2\leq
N/(\alpha+1)$ which is equivalent to $\alpha\leq \alpha_\star$. This
is in contradiction with the hypothesis on $\alpha$. Therefore
$u\equiv
0$ in $\Omega$.\\

\noindent \textsl{The case $\alpha=\alpha_\star$.}\\

 From Lemma \ref{lem2}, due to the criticality of $\alpha_\star$,
 $u$ must satisfy
$$\int_{\p \Omega} (x\cdot \nu)\Big(\frac{\p u}{\p \nu}\Big)^2 d\sigma=0.$$ We fix  $\Omega=\{x\in
\rr_{+}^{N} \ |\ |x'|^2+(x_N-1)^2\leq 1\}$ which is star-shaped.
 Therefore,
$$\frac{\p u}{\p \nu}=0, \q \textrm{a.e. on }\Gamma.$$
 Therefore, the problem in consideration is
reduced to the overdetermined system
\begin{equation}\label{overlap}
\left\{ \begin{array}{ll}
  -\D u-\frac{\lambda}{|x|^2}u=|u|^{\frac{4}{N-2}} u,   & x\in  \Omega,  \\
  u=0, & x\in \Gamma, \\
  \frac{\p u}{\p \nu}=0, & x\in \Gamma.
\end{array}\right.
\end{equation}

Let us consider a compact subset $\Gamma'\subset \Gamma$ such that
$x\cdot \nu >0$ and $0\not \in \Gamma'$. Next,  we extend $\Omega$
with a bounded set $\Omega_1$ such that $\Omega_1 \cap \Omega={\O}$,
$\p \Omega_1 \cap \p \Omega=\Gamma'$,
$\tilde{\Omega}:=\Omega\cup \Omega_1$.\\
 For $\eps>0$ small enough
we denote the sets $\Omega_\eps:=\Omega\setminus \{ x\in \Omega | \
|x|< \eps\}$, $\tilde{\Omega}_\eps:=\tilde{\Omega}\setminus \{x\in
\Omega\ |\ |x|< \eps \}$.

Consider also the trivial prolongation of u to $\tilde{\Omega}$
\begin{equation}\label{eqq1}
\tilde{u}:=\left\{\begin{array}{ll}
  u, & x\in \Omega, \\
  0, &  \Omega_1.\\
\end{array}\right.
\end{equation}
The fact that $u\in D(A_\la)$ combined with the over-determined
condition in \eqref{overlap}, imply that $u\in
H^2(\tilde{\Omega}_\eps)$. Let us also show that $\tilde{u}\in
H^2(\tilde{\Omega}_\eps)$.

Indeed, thanks to \eqref{overlap} on $\Gamma_0$ we get that
\begin{equation}\label{deriv}
\int_{\tilde{\Omega}_\eps} \frac{d \tilde{u}}{\p x_i} \frac{\p
\phi}{\d x_j} dx =-\int_{\tilde{\Omega}_\eps} g \phi dx, \q \forall
\phi \in C_{0}^{\infty}(\tilde{\Omega}_\eps),
\end{equation}
where $g\in L^2(\tilde{\Omega}_\eps)$ is given by
\begin{equation}
g=\left\{
\begin{array}{ll}
  \frac{\p^2 u}{\p x_i \p x_j},   &  x\in \Omega_\eps, \\
  0, & x\in \Omega_1. \\
\end{array}\right.
\end{equation}

In particular we obtain that
\begin{equation}
\D \tilde{u}=\left\{
\begin{array}{ll}
  \D u,   &  x\in \Omega_\eps, \\
  0 & x\in \Omega_1. \\
\end{array}\right.
\end{equation}
and $\tilde{u}$ verifies
\begin{equation}\label{neweq}
-\D
\tilde{u}-\frac{\la}{|x|^2}\tilde{u}=|\tilde{u}|^{\frac{4}{N-2}}\tilde{u}
\ \textrm{ a.e.  in } \tilde{\Omega}_\eps
\end{equation}
and $\tilde{u}\equiv 0$ in $\Omega_1$. In other words we can write
\eqref{neweq} as
$$-\D \tilde{u}=V(x) \tilde{u}, \q x\in \tilde{\Omega}_\eps,$$
where $V(x)=\frac{\la}{|x|^2}+|\tilde{u}|^{\frac{4}{N-2}}$. Note
that $V\in L^\omega(\tilde{\Omega}_\eps)$ for some $\omega> N/2$ and
$\tilde{u}$ vanishes in $\Omega_1$.

With these we are in the hypothesis of the strong unique
continuation result by Jerison and Kenig \cite{kenig}.  Therefore,
$\tilde{u}\equiv 0$ in $\tilde{\Omega}_\eps$ and in particular
$u\equiv 0$ in $\Omega_\eps$, for any $\eps>0$. Hence, we conclude
that $u\equiv 0$ in $\Omega$. The proof of Theorem \ref{th1} is
finished.
\end{proof}

\begin{proof}[Proof of Lemma \ref{lem2}]
Since $u\in D(A_\la)$ we can apply the Pohozaev identity and we get
\begin{equation}\label{poh}
\frac{1}{2}\int_{\p\Omega}( x\cdot \nu)\Big(\frac{\p u}{\p
\nu}\Big)^2
 d\sigma
=\into -|u|^{\alpha-1} u (x\cdot \n u)
dx-\frac{N-2}{2}||u||_{H_\la}^{2}, \end{equation} Next we show that
\begin{equation}\label{cri1}
\into |u|^{\alpha} u (x\cdot \n u)dx=-\frac{N}{1+\alpha}\into
|u|^{\alpha+1}dx.
\end{equation}
We proceed by approximation arguments. For $\eps>0$  small enough we
 consider $I_\eps:=\into |u|^{\alpha} u (x\cdot \n u)\theta_\eps
dx$, where $\theta_\eps$ is a cut-off function supported in
$\Omega\setminus B_{\eps}(0)$.  Due to the fact that $u\in
H^2(\Omega\setminus \{0\})$ we can integrate by parts as follows.
\begin{align}\label{aprox}
I_\eps&= -\frac{1}{2}\into |u|^{p-1} x\cdot \n (u^2) \theta_\eps dx
= \frac{1}{2}\into u^2 \textrm{div}\big[ |u|^{\alpha-1} x
\theta_\eps\big]dx \nonumber\\
&=\frac{1}{2} \into u^2 \big[N |u|^{\alpha-1}\theta_\eps +x\cdot
\theta_\eps  |u|^{\alpha-1} + (\alpha-1) x\cdot \n u
|u|^{\alpha-3}u\big]dx\nonumber\\
&= \frac{N}{2}\into |u|^{\alpha+1} \theta_\eps dx + \frac{1}{2}\into
|u|^{\alpha+1} x\cdot \n \theta_\eps dx -\frac{\alpha-1}{2}I_\eps.
\end{align}
Therefore we obtain
\begin{equation}\label{appr}
I_\eps= \frac{N}{\alpha+1} \into |u|^{\alpha+1} \theta_\eps dx+
\frac{1}{\alpha+1}\into |u|^{\alpha+1} x\cdot \n \theta_\eps dx.
\end{equation}
From the equation itself it is easy to see that $|u|^{\alpha+1}\in
L^1(\Omega)$ provided $u\in D(A_\la)$. Therefore, by the DCT we can
pass to the limit as $\eps\rightarrow 0$ in \eqref{appr} to obtain
the identity \eqref{cri1}. On the other hand, multiplying
\eqref{nonexist} by $u$ and integrating we obtain
$$||u||_{H_\la}^2=\into |u|^{\alpha+1}dx,$$
 Combining this with  \eqref{cri1}
 and \eqref{poh} we conclude \eqref{negative}.
\end{proof}

\subsection{Brief presentation of
the case C2}\label{2sec} Inequalities \eqref{oeq3}, \eqref{oeq44}
can be  stated in a simplified form as follows.

 Assume $\Omega\subset \rr^N$ is a smooth bounded domain containing the origin on the boundary.
  For any $\l \leq  N^2/4$   and any $0<\gamma<2$ 
  there exists a constant $C_1(\gamma, \Omega)\geq 0$ such that
\begin{equation}\label{Hardy}
    \forall u\in \hoi,
    \qquad     \int_{\Omega} \frac{u^2}{|x|^\gamma}dx
    + \l \into
    \frac{u^2}{|x|^2} dx \leq \into |\nabla u|^2+
    C_1(\gamma,
\Omega)\into u^2 dx. \end{equation}

\subsubsection{ Functional framework via Hardy inequality}

Let us now define the set
\begin{equation}\label{set}
 \mathcal{C}:= \Big \{ C\geq 0 \ \textrm{ s. t. } \ \inf_{u\in \hoi}
 \frac{\into \big[|\n u|^2-\l(N)u^2/|x|^2+C u^2\big ] dx
 }{\into
 u^2/|x|^\gamma dx} \geq 1 \ \Big\}.
\end{equation}

Of course, $\mathcal{C}$ is non empty due to inequality
\eqref{Hardy}. Next we define
\begin{equation}\label{optimalconstant}
\mathcal{C}_0=\inf_{C\in \mathcal{C}}C .
\end{equation}

Then, for any  $\la\leq \la(N)=N^2/4$ we introduce the Hardy
functional
\begin{equation}\label{eq130}
B_\la[u]:=\int_{\Omega} |\nabla u|^2dx-\la\int_{\Omega}
\frac{u^2}{|x|^2}dx+\mathcal{C}_0\int_{\Omega} u^2dx,
\end{equation}
which is positive for any $u\in \hoi$ due to inequality
\eqref{Hardy} and the election of $\mathcal{C}_0$. Then we define
the corresponding Hilbert space $H_{\la}$ as the closure of
$C_{0}^{\infty}(\Omega)$ in the norm induced by $B_\la[u]$. Observe
that for any $\la< \la(N)$ the identification $H_\la =\hoi$ holds
true. Indeed, if $\la< \la(N)$, we have
\begin{align}\label{normsubcritical}
B_\la[u]&\geq \big(1-\frac{\la}{\la(N)}\big)\into |\n u|^2
dx-\frac{\mathcal{C}_0 \la}{\la(N)}\into u^2 dx.
\end{align}
On the other hand, from the definition of $\mathcal{C}_0$ we obtain
that there exists a constant $C_1=C_1(\gamma)>0$ such that
\begin{align}\label{normsub}
B_\la[u] \geq C_1 \into u^2 dx.
\end{align}
Multiplying \eqref{normsub} by $\mathcal{C}_0 \la /(C_1 \la(N))$ and
summing to \eqref{normsubcritical} we get that
$$B_\la[u] \geq C_\la \into |\n u|^2 dx,$$
for some positive constant $C_\mu$ that converges to zero as $\la$
tends to $\la(N)$.

Besides, in the critical case $\la=\la(N)$, $H_\la$ is slightly
larger than $\hoi$. However,  using cut-off arguments near the
singularity (see e.g. \cite{vazzua}) we can show that
\begin{equation}\label{slightly}
B_\la[u]_{\la(N)}\geq C_\eps ||u||_{H^{1}(\Omega \setminus
B_\eps(0))}, \q \forall u\in \hoi
\end{equation}
where $C_\eps$ is a constant going to zero as $\eps$ tends to zero.

Let us define de operator $A_\la:=-\D-\la/|x|^2+\mathcal{C}_0 I$ and
define its domain as
\begin{equation}\label{domain}
D(A_\la):=\{ u \in H_\la \ | \  A_\la u \in L^2(\Omega)\}.
\end{equation}
The norm of the operator $A_\la$ is given by
\begin{equation}\label{operator}
||u||_{D(A_\la)}=||u||_{L^2(\Omega)}+ ||A_\la u||_{L^2(\Omega)}.
\end{equation}

\subsubsection { The meaning of the  $H_\la$-norm}

First of all we remark that
\begin{equation}\label{genident}
\into |\n u|^2 dx+ \into \frac{\D \Phi}{\Phi}u^2 dx = \into \Big| \n
u -\frac{\n \Phi}{\Phi}u\Big|^2 dx, \q \forall \ u\in
C_{0}^{\infty}(\Omega),
\end{equation}
and any distribution satisfying $\Phi, 1/\Phi \in
C^1(\Omega\setminus\{0\}) $ and $\Phi>0$ in $\Omega$.

Let us also consider $\phi(x)=\phi(|x|)\in C^\infty(\Omega)$ to be a
cut-off function such that
\begin{equation}\phi=\left\{
\begin{array}{ll}
1,& |x|\leq r_0/2, \ x\in \Omega\\
0, &  |x|\geq r_0, \ x\in \Omega,\\
\end{array}\right.
\end{equation}
where $r_0>0$ is aimed to be small.

\noindent \textsl{ Case 1.Assume the points on the boundary $\Gamma$
of $\Omega$ satisfy  $x_N>0$ in a neighborhood of the origin.}

We take $\Phi_1=x_N|x|^{-N/2}$ which satisfies the equation
\begin{equation}\label{eqdi} -\D \Phi_1
-\frac{N^2}{4}\frac{\Phi_1}{|x|^2}=0, \q \textrm{ a. e.  in
}\Omega_{r_0}, \end{equation}
 where $\Omega_{r_0}:=\Omega\cap
B_{r_0}(0)$ for some $r_0>0$ small enough. From \eqref{genident} and
\eqref{eqdi} we obtain
\begin{equation}\label{partial}
\int_{\Omega_{r_0}}|\n v|^2dx -\frac{N^2}{4} \int_{\Omega_{r_0}}
\frac{v^2}{|x|^2}dx=\int_{\Omega_{r_0}} \big| \n v-\frac{\n
\Phi_1}{\Phi_1}v\big|^2 dx, \q \forall \  v \in
C_{0}^{\infty}(\Omega_{r_0}).
\end{equation}
 By a standard cut-off argument, due to \eqref{partial}
 we remark that, there exist some
  weights $\rho_1,\rho_2\in C^\infty(\Omega)$ depending on $r_0$,
  supported far from origin
 such  that
\begin{align}\label{equivnorm}
B_\la[u]&=\into \Big|\n (u\phi)-\frac{\n \Phi_1}{\Phi_1} (u\phi)
 \Big|^2 dx+\into \rho_1|\n u|^2 dx\nonumber\\
&+(\la(N)-\la)\into \frac{u^2}{|x|^2} dx +   \into \rho_2 u^2 dx, \q
\forall  u\in C_{0}^{\infty}(\Omega).
\end{align}
Then the meaning of $||\cdot||_{H_\la}$-norm  is
characterized by 
\begin{align}\label{newnorm}
||u||_{H_\la}^2&= \lim_{\eps\rightarrow 0}\int_{x\in \Omega,
|x|>\eps}\Big|\n (u\phi)-\frac{\n \Phi_1}{\Phi_1}(u\phi)
\Big |^2 dx+\into \rho_1|\n u|^2 dx\nonumber\\
&+(\la(N)-\la)\into \frac{u^2}{|x|^2} dx +   \into \rho_2 u^2 dx, \q
 \forall u \in H_\la, \q \forall \la\leq \la(N).
\end{align}

\noindent \textsl{ Case 2. Assume the points on $\Gamma$ satisfy
$x_N\leq 0$ in a neighborhood of the origin}

In this case we consider  $d=(x, \Gamma)=d(x)$ the function denoting
the distance from a point $x\in \Omega$ to  $\Gamma$. We  remark
that close enough to origin the distribution
$$\Phi_2= d(x)e^{(1-N)d(x)}|x|^{-N/2}\Big|\log\frac{1}{|x|}\Big|^{1/2},$$
satisfies
$$ -\D \Phi_2 -\frac{N^2}{4|x|^2}\Phi_2=P>0,  \q \forall
  x\in \Omega_{r_0}$$
where $r_0>0$ is small enough.%
 Due to this, there exist the weights  $ \rho_1, \rho_2\in
C^\infty(\Omega)$ depending on $r_0$ and supported away from origin,
such that the meaning of $H_\la$-norm is given by
\begin{align}\label{newnorm2}
||u||_{H_\la}^2&=\lim_{\eps\rightarrow 0}\int_{x\in \Omega,
|x|>\eps} \Big| \n (u\phi)- \frac{\n \Phi_2}{\Phi_2}(u\phi)  \Big|^2
dx+\into \frac{P}{\Phi_2} |u\phi|^2 dx \nonumber\\&
+(\la(N)-\la)\into \frac{u^2}{|x|^2}+\into \rho_1|\n u|^2 dx +\into
\rho_2 u^2 dx, \q \forall u\in H_\la, \q \forall \la\leq \la(N).
\end{align}

\noindent \textsl{ Case 3. Assume that $x_N$ changes sign on
$\Gamma$ at the origin.}

This case can be analyzed through  Case 2 above.\\

Then, the Pohozaev identity and  related results presented in case
C1 might be extended to case C2 by means of the weaker functional
settings introduced above.

\section{Applications to Controllability}\label{5sec}
 In this section we study the controllability of
  the wave and Schr\"{o}dinger equations with
singularity localized on the boundary of a smooth domain. Our
motivation came through the results shown in \cite{judith} in the
context of interior singularity.

For the sake of clarity, we
will discuss in a detailed manner the case C1. 

\subsection{ The wave equation. Case C1}\label{9sec}
In the sequel,  we are  focused to the controllability of the
wave-like system
\begin{equation}\label{eqW}(W_{\la}):\left\{\begin{array}{ll}
  u_{tt}-\Delta u-\lambda \frac{u}{|x|^2}=0, & (t,x)\in Q_T, \\
  u(t,x)=h(t,x), & (t,x)\in (0,T)\times \Gamma_0, \\
  u(t,x)=0, & (t,x)\in (0,T)\times (\Gamma\setminus \Gamma_0), \\
  u(0,x)=u_0(x), &  x\in \Omega, \\
  u_t(0,x)=u_1(x), & x \in \Omega. \\
\end{array}\right.
\end{equation}
 where $Q_T=(0, T)\times \Omega $, $\Gamma$ denotes
 the boundary of $\Omega$ and $\Gamma_0$
 is the boundary control region  defined in \eqref{oeq2}, where
  the control $h \in L^2((0, T)\times
 \Gamma_0)$ is acting. We also assume
 $\la\leq \la(N)$.
 In view of the
time-reversibility of the equation it is enough to consider the case
where the target
$$(\overline{u_0},\overline{u_1})=(0,0).$$ It is the so-called \textsl{null
controllability problem}.

\subsubsection{Well-posedness}\label{oeq6}
 Let us briefly  discuss the
well-posedness of system \eqref{eqW} in the corresponding functional
setting.

 Instead of (\ref{eq124}) we firstly consider the more general
system with non-homogeneous boundary conditions:
\begin{equation}\label{eq62}\left\{\begin{array}{ll}
  u_{tt}-\Delta u-\lambda \frac{u}{|x|^2}=0, & (t,x)\in Q_T, \\
  u(t,x)=g(t,x), & (t,x)\in  \Sigma_T, \\
  u(0,x)=u_0(x), &  x\in \Omega, \\
  u_t(0,x)=u_1(x), & x \in \Omega. \\
\end{array}\right.
\end{equation}
where $\Sigma_T=(0, T)\times \Gamma$.
 The solution of (\ref{eq62}) is defined by the transposition
 method (J.L. Lions \cite{lions1}):
\begin{Def}\label{d1}
Assume $\la\leq \la(N)$. For $(u_0,u_1)\in L^2(\Omega)\times
H'_{\lambda}$ and $g\in L^2((0,T)\times \Gamma)$, we say that $u$ is
a {\bf weak solution} for (\ref{eq62}) if
\begin{equation}\label{eq63}
\int_{0}^{T}\int_{\Omega} u f dx=-<u_0,z'(0)>_{L^2(\Omega),
L^2(\Omega)}+<u_1,z(0)>_{H_{\lambda}^{'},H_\lambda}
-\int_{0}^{T}\int_{\Gamma}g\frac{\p z}{\p \nu} \quad \forall \q f\in
\mathcal{D}(\Omega),
\end{equation}
where $<\cdot,\cdot>$ represents the dual product between
$H_\lambda$ and its dual $H_{\lambda}^{'}$, and $z$ is the solution
of the non-homogeneous adjoint-backward problem
\begin{equation}\label{eq16}\left\{\begin{array}{ll}
  z_{tt}-\Delta z-\lambda \frac{z}{|x|^2}=f, & (t,x)\in Q_T, \\
  z(t,x)=0, & (t,x)\in \Sigma_T, \\
  z(T,x)=z'(T,x)=0, &  x\in \Omega.
\end{array}\right.
\end{equation}
\end{Def}
Formally, \eqref{eq63} is obtained by multiplying the system
\eqref{eq16} with $u$  and integrate on $Q_T$. 
 Using the Hardy inequalities above and the application of standard
methods for evolution equations we lead to the following existence
result.
\begin{Th}[{\bf well-posedness}]\label{t2}
Assume that $\Omega$ satisfies  C1. Let $T>0$ be given and assume
$\lambda\leq \lambda(N)$. For every $(u_0,u_1)\in L^2(\Omega)\times
H_{\lambda}^{'}$ and any $h\in L^2((0,T)\times \Gamma_0)$ there
exists a unique weak solution of (\ref{eq124}) such that
\begin{equation}\label{eq20}
u\in C([0,T]; L^2(\Omega))\cap C^1([0,T]; H_{\lambda}^{'}).
\end{equation}
Moreover,  the solution of (\ref{eq124}) satisfies
\begin{equation}\label{eq21}
||(u,u_t)||_{L^\infty(0,T; L^2(\Omega)\times
H_{\lambda}^{'})}\lesssim ||(u_0,u_1)||_{L^2(\Omega)\times
H_{\lambda}^{'}}+||h||_{L^2((0,T)\times \Gamma_0)}.
\end{equation}
\end{Th}

The details of the proof of Theorem \ref{t2} are omitted since they
follow the same steps as in \cite{judith}.

\subsubsection{Controllability and main results}\label{osec1} 
It is by now classical that controllability of  \eqref{eqW} is
characterized through an observability inequality for the adjoint
system as follows below.

  Given initial data  $(u_0,u_1)\in
L^2(\Omega)\times {H_\lambda}'$, 
a possible  control $h\in L^2((0,T)\times \Gamma_0)$ must satisfy
the identity
\begin{equation}\label{eq6}
\int_{0}^{T}\int_{\Gamma_0}h\ \frac{\partial v}{\partial \nu}d\sigma
dt-<u_t(0), v(0)>_{H_{\lambda}^{'}, H_\lambda}+<u(0),
v_t(0)>_{L^2(\Omega), L^2(\Omega)}=0,
\end{equation}
where $v$ is the solution of the adjoint system
\begin{equation}\label{eq2}\left\{\begin{array}{ll}
  v_{tt}-\Delta v-\lambda \frac{v}{|x|^2}=0, & (t,x)\in Q_T, \\
  v(t,x)=0, & (t,x)\in  \Sigma_T, \\
  v(0,x)=v_0(x), &  x\in \Omega, \\
  v_t(0,x)=v_1(x), & x \in \Omega. \\
\end{array}\right.
\end{equation}
The operator $\mathcal{A_\la}$ defined by $\mathcal{A_\la}(w_1,w_2)
= (w_2, \D w_1 + \la |x|^2w_1)$ for all $(w_1,w_2) \in
D(\mathcal{A_\la}) = D(A_\la) \times  H_\la$, generates the wave
semigroup i.e. $ (\mathcal{A_\la},D(\mathcal{A_\la}))$ is
m-dissipative in $H_\la\times L^2(\Omega)$. In view of that, due to
the theory of semigroups,  the adjoint system is well-posed and more
precisely it holds that

\begin{prop}[see, e.g.\cite{judith}]\label{p3}
\begin{enumerate}[(1)]
\item For any initial data $(v_0,v_1)\in H_\lambda\times
L^2(\Omega)$  there exists a unique solution of (\ref{eq2}) $$u \in
C([0,T]; H_\lambda)\cap C^1([0,T]; L^2(\Omega)).$$ Moreover,
\begin{equation}\label{eq18}
||(v ,v_t)||_{L^\infty(0,T; H_\lambda\times L^2(\Omega))}\lesssim
||v_0||_{H_\lambda}+||v_1||_{L^2(\Omega)}
\end{equation}
\item For any initial data $(v_0,v_1)\in D(A_\la)\times
H_\lambda$  there exists a unique solution of (\ref{eq2}) such thar
$$v \in C([0,T]; D(A_\la))\cap C^1([0,T]; H_\lambda)\cap C^2([0,T]; L^2(\Omega)). $$
Moreover
\begin{equation}\label{eq18}
||(v,v_t)||_{L^\infty(0,T; D(A_\la)\times H_\lambda)}\lesssim
||v_0||_{D(A_\la)}+||v_1||_{H_\lambda}
\end{equation}
\end{enumerate}
\end{prop}

In the sequel, we claim some ``hidden regularity" effect for the
system \eqref{eq2} which may not be directly deduce from the
semigroup regularity but from the equation itself.
\smallskip
\begin{Th}[\bf Hidden regularity]\label{Wmult}
Assume $\la \leq \la(N)$ and $v$ is the solution of \eqref{eq2}
corresponding to the initial data $(v_0, v_1)\in H_\la\times
L^2(\Omega)$. Then $v$ satisfies
\begin{equation}\label{ineqtrace1}
\int_{0}^{T}\int_{\Gamma} (x\cdot \nu)\Big(\frac{\p v }{\p
\nu}\Big)^2d\sigma dt\lesssim \int_{0}^{T}\int_{\Gamma}\Big(\frac{\p
v }{\p \nu}\Big)^2|x|^2 d\sigma dt \lesssim
||v_0||_{H_\la}^2+||v_1||_{L^2(\Omega)}^2.
\end{equation}
 Moreover, $v$ verifies the identity
\begin{equation}\label{multipliers}
\frac{1}{2}\int_{0}^{T}\int_{\Gamma} (x\cdot \nu)\Big(\frac{\p v
}{\p \nu}\Big)^2d\sigma dt =\frac{T}{2}
(||v_0||_{H_\la}^2+||v_1||_{L^2(\Omega)}^2) +\int_{\Omega} v_t
\big(x\cdot \n v+\frac{N-1}{2}v\big)\Big|_{0}^{T}dx.
\end{equation}
\end{Th}

Due to Theorem \ref{Wmult} the operator $(v_0, v_1)\mapsto
\big(\int_{0}^{T}\int_{\Gamma_0}(x\cdot \nu)(\p v/ \p \nu)^2d\sigma
dt\big)^{1/2}$ 
 is a linear continuous map in $H_\la\times L^2(\Omega)$. Let
$\mathcal{H}$ be the completion of this norm in $H_\la \times
L^2(\Omega)$.  We consider the functional $J: \mathcal{H}\rightarrow
\rr$ defined by
\begin{equation}\label{functional}
J(v_0, v_1)(v):=\frac{1}{2}\int_{0}^{T}\int_{\Gamma_0}(x\cdot
\nu)\Big(\frac{\partial v} {\partial \nu}\Big)^2d\sigma dt
-<u_1,v_0>_{H_{\lambda}^{'},H_\lambda}+(u_0,v_1)_{L^2(\Omega),L^2(\Omega)},
\end{equation}
where $v$ is the solution of \eqref{eq2} corresponding to initial
data $(v_0, v_1)$. Of course,
$<\cdot,\cdot>_{H_{\lambda}^{'},H_{\lambda}}$ denotes the duality
product.  A control $h\in L^2((0, T)\times \Gamma_0)$ satisfying
\eqref{eq6}  could be chosen as $h=(x\cdot \nu) v_{\min}$ where
$v_{\min}$ minimizes the functional $J$ on $\mathcal{H}$ among the
solutions $v$ of (\ref{eq2}) corresponding to the initial data
$(u_0, u_1)\in H_{\lambda}^{'}\times L^2(\Omega)$
 The existence of a minimizer of $J$  is assured by the
coercivity of $J$, which is equivalent to
 the \textsl{Observability inequality} for the adjoint
system \eqref{eq2}: 
\begin{equation}
||v_0||_{H_\la}^2+||v_1||_{L^2(\Omega)}^2\lesssim \int_{0}^{T}
\int_{\Gamma_0} (x\cdot \nu) \Big( \frac{\p v}{\p \nu}\Big)^2
d\sigma dt,
\end{equation}

\noindent \textsl{Conservation of energy.}

\noindent For any $\lambda \leq \lambda(N)$ and any fixed time
$t\geq 0$, let us define the energy associated to (\ref{eq2}):
\begin{equation}
E_{v}^{\lambda}(t)=\frac{1}{2}\big(||v_t(t)||_{L^2(\Omega)}^{2}+
||v(t)||_{H_\lambda}^{2}\big)
\end{equation}
We note that our system is conservative and therefore
$$E_{v}^{\la}(t)=E_{v}^{\la}(0), \q \forall \la\leq \la(N), \q \forall
t\in [0, T].$$ Next we claim our main results which answer to the
controllability question.

\begin{Th}[{\bf Observability inequality}]\label{t1}
For all $\lambda\leq \lambda(N)$, there exists a positive constant
$D_1=D_1( \Omega, \la , T)$ such that for all $T\geq 2R_\Omega$, and
any initial data $(v_0, v_1)\in H_\la \times L^2(\Omega)$ the
solution of (\ref{eq2}) verifies the observability inequality
\begin{equation}\label{ObsIneq1}
E_{v}^{\l}(0)\leq D_1\int_{0}^{T}\int_{\Gamma_0}(x\cdot
\nu)\Big(\frac{\p v}{\p \nu}\Big)^2d\sigma dt.
\end{equation}
\end{Th}

The proof of Theorem \ref{t1} relies mainly on the method of
multipliers (cf. \cite{lions1}) and the so called compactness
uniqueness argument (cf. \cite{mach}), combined with the new Hardy
inequalities above. These results guarantee the exact
controllability of (\ref{eq124}) when the control acts on the part
$\Gamma_0$. In conclusion, we obtain
\begin{Th}[{\bf Controllability}]\label{ht1}
Assume that $\Omega$ satisfies $\textsc{C}\ref{eqq2}$ and
$\lambda\leq \lambda(N)$. For any time $T>2R_{\Omega}$,
$(u_0,u_1)\in L^2(\Omega)\times H_{\lambda}^{'}$ and
$(\overline{u_0}, \overline{u_1}) \in L^2(\Omega)\times
H_{\lambda}^{'}$ there exists $h\in L^2((0,T)\times\Gamma_0)$ such
that the solution of (\ref{eq124}) satisfies
\begin{equation*}(u_t(T,x),u(T,x))=(\overline{u_1}(x),\overline{u_0}(x)) \quad
\textrm{ for all } x\in \Omega.
\end{equation*}
\end{Th}

\subsubsection{Proofs of main results}
First of all, we need to justify that the solution $v$ of adjoint
system \eqref{eq2} posses enough regularity to guarantee the
integrability of the boundary term in  \eqref{ObsIneq1}.  The
justification is not trivial because  the presence of the
singularity at the boundary.

\begin{proof}[Proof of Theorem \ref{Wmult}]
We will proceed straightforward from Theorem \ref{poho}.

Firstly,  we consider  initial data $(v_0, v_1)$ in
$D(\mathcal{A_\la})=D(A_\la)\times H_\la$. Then, according to
Proposition \ref{p3} we have
$$v \in C([0,T]; D(A_\la))\cap C^1([0,T]; H_\la)\cap C^2([0,T];
L^2(\Omega)).$$
   For a fixed time $t\in [0, T]$ we apply
    Theorem \ref{poho} for $A_\la v=-v_{tt}$ and we obtain
\begin{equation}
\int_{\Gamma} \Big(\frac{\p v}{\p \nu}(x, t)\Big)^2|x|^2d\sigma
\lesssim ||v(t)||_{H_\la}^2+||v_{tt}(t)||_{L^2(\Omega)}^2, \q
\forall t\in [0,T].
\end{equation}
Integrating in time and if necessary in space we derive
\begin{align}
\int_{0}^{T}\int_{\Gamma} \Big(\frac{\p v}{\p \nu}(x,
t)\Big)^2|x|^2d\sigma dt \lesssim \int_{0}^{T}||v(t)||_{H_\la}^2 dt
+\int_{\Omega} v_t^2(T,x)dx-\into v_t^2(0, x) dx.
\end{align}
According to the conservation of energy we reach to
\begin{align}
\int_{0}^{T}\int_{\Gamma} \Big(\frac{\p v}{\p \nu}(x,
t)\Big)^2|x|^2d\sigma dt
&\lesssim 2\int_{0}^{T} E_v^\la(t) dt+ E_v^\la(T)\nonumber\\
& =(2T+2)
E_v^\la(0)\nonumber\\
&=(T+1)(||v_0||_{H_\la}^2+||v_1||_{L^2(\Omega)}^2).
\end{align}
Since $x\cdot \nu\lesssim |x|^2$ on $\Gamma$, from above we conclude
the inequality \eqref{ineqtrace1}.

 Next, we apply the Pohozaev identity for $v(t)$,  $t\in [0, T]$.
  Indeed, integrating
in time in Theorem \ref{poho} for $A_\la v=-v_{tt}$,  we get
\begin{align}
\frac{1}{2}\int_{0}^{T}\int_{\Gamma} (x\cdot \nu)\Big(\frac{\p v}{\p
\nu}\Big)^2d\sigma dt&= 
 \into v_t (x\cdot \n v)\Big|_{0}^{T} dx-\frac{N-2}{2}\int_{0}^{T}
||v(t)||_{H_\la}^2 dt\nonumber\\
&=\into v_t (x\cdot \n v)\Big|_{0}^{T} dx-\frac{1}{2}\intot x \cdot
\n (v_t^2) dx dt -\frac{N-2}{2}\int_{0}^{T} ||v(t)||_{H_\la}^2
dt\nonumber\\
&= \into v_t (x\cdot \n v)\Big|_{0}^{T} dx+\frac{N}{2}\into
||v_t(t)||_{L^2(\Omega)}^2 dt -\frac{N-2}{2}\int_{0}^{T}
||v(t)||_{H_\la}^2 dt\nonumber\\
&=\into v_t (x\cdot \n v)\Big|_{0}^{T}
dx+\frac{1}{2}\int_{0}^{T}\big[ ||v_t(t)||_{L^2(\Omega)}^2 dt+
||v(t)||_{H_\la}^2\big]dt\nonumber\\
&+\frac{N-1}{2}\int_{0}^{T}\big[||v_t(t)||_{L^2(\Omega)}^2-
||v(t)||_{H_\la}^2 \big]dt.
\end{align}
Multiplying the equation of $(W_\la)_{adj}$ by $v$ and integrate in
space,
 the equipartition of the energy
$$\into v v_t\Big|_{0}^{T} dx=||v_t(t)||_{L^2(\Omega)}^{2}
-||v||_{H_\la}^2,$$ holds true. Due to the conservation of energy
and from relations above,  we obtain precisely the  identity
\eqref{multipliers}. This yields the proof of Theorem \ref{trace}
for initial data in the domain $D(\mathcal{A_\la})$.
 Then, by density arguments, one can extend
the results for less regular initial data $(v_0, v_1)\in H_\la
\times
L^2(\Omega)$. 
\end{proof}

\begin{proof}[Proof of Theorem \ref{t1}]

In what follows we present the proof in the critical case
$\la=\la(N)$, which is of main interest. The subcritical case
$\la<\la(N)$ is let to the reader.

{\bf Step 1.}  Firstly, from Lemma \ref{Wmult} we remark  that
\begin{equation}\label{eq14}
\int_{\Omega}v_t(\frac{N-1}{2}v+x\cdot \nabla v)dx\Big|_{0}^{T}+
TE_{v}^{\lambda(N)}(0)\leq
\frac{1}{2}\int_{0}^{T}\int_{\Gamma_0}(x\cdot \nu)(\frac{\p v}{\p
\nu})^2d\sigma dt.
\end{equation}
For a fixed time $t=t_0>0$, by Cauchy-Schwartz inequality we have
\begin{align*}
\Big|\int_{\Omega}v_t\Big(\frac{N-1}{2}v+x\cdot \nabla
v\Big)dx\Big|_{t=t_0}&\leq
\frac{R_\Omega}{2}\int_{\Omega}v_{t}^{2}dx+\frac{1}{2R_{\Omega}}\int_{\Omega}\Big(\frac{N-1}{2}v+x\cdot
\nabla v\Big)^2dx\nonumber\\
&=
\frac{R_\Omega}{2}||v_t||_{L^2(\Omega)}^{2}+\frac{1}{2R_\Omega}\Big[\Big(\frac{N-1}{2}\Big)^2||v||_{L^2(\Omega)}^{2}+||x\cdot
\nabla v||_{L^2(\Omega)}^{2}\nonumber\\
&+(N-1)\int_{\Omega}v(x\cdot \nabla v)dx\Big].
\end{align*}
On the other hand it follows
 $$\int_{\Omega}v(x\cdot \nabla
v)dx=\frac{1}{2}\int_{\Omega}x\cdot \nabla
(v^2)dx=-\frac{1}{2}\int_{\Omega}\textrm{div}(x)v^2dx=-\frac{N}{2}\int_{\Omega}v^2dx$$
Therefore we obtain
\begin{align*}\label{eq15}
\Big|\int_{\Omega}v_t\Big(\frac{N-1}{2}v+x\cdot \nabla
v\Big)dx\Big|_{t=t_0}&\leq
\frac{1}{2R_\Omega}||x\cdot \nabla
v||_{L^2(\Omega)}^{2}+\frac{R_\Omega}{2}||v_t||_{L^2(\Omega)}^{2}-\frac{1}{2R_\Omega}\Big(\frac{N^2-1}{4}\Big)||v||_{L^2(\Omega)}^{2}
\end{align*}
Applying Theorem \ref{t1}  we deduce
\begin{equation}\label{key}
\Big|\int_{\Omega}v_t\Big(\frac{N-1}{2}v+x\cdot \nabla
v\Big)dx\Big|_{t=t_0}\leq R_\Omega
E_{v}^{\lambda(N)}(t_0)-C||v(t_0)||_{L^2(\Omega)}^{2},
\end{equation}
 for some constant $C$. Due to the conservation of the
energy and taking $t_0=0$ respectively $t_0=T$ and summing in
\eqref{key} we get
\begin{equation}\label{eq24}
\Big|\int_{\Omega}v_t\Big(\frac{N-1}{2}v+x\cdot \nabla
v\Big)dx\Big|_{t=0}^{t=T}\Big|\leq 2R_\Omega E_{v}^{\lambda(N)}(0)-
C(||v(0)||_{L^2(\Omega)}^{2}+||v(T)||_{L^2(\Omega)}^{2}),
\end{equation}
   From (\ref{eq14}) and (\ref{eq24}) we obtain
\begin{equation}\label{eq25}
(T-2 R_\Omega)E_{v}^{\lambda(N)}(0)\leq
\frac{1}{2}\int_{0}^{T}\int_{\Gamma_0}(x\cdot \nu)\Big(\frac{\p
v}{\p \nu}\Big)^2d \sigma
dt+C(||v(0)||_{L^2}^{2}+||v(T)||_{L^2}^{2}).
\end{equation}
{\bf Step 2. } To get rid of the remainder term at the right hand
side of (\ref{eq25}) we need the following lemma.
\begin{lema}\label{l6}
There exists a positive constant $C=C(T, \Omega)>0$ such that
\begin{equation}\label{eq26}
||v(0)||_{L^2(\Omega)}^{2}+||v(T)||_{L^2(\Omega)}^{2}\leq
C\int_{0}^{T}\int_{\Gamma_0}(x\cdot \nu)\Big(\frac{\p v}{\p
\nu}\Big)^2d \sigma dt
\end{equation}
for all finite energy solution of (\ref{eq62}).
\end{lema}
Combining Lemma \ref{l6} with \eqref{eq25}, the Observability
inequality is finally proved.
\end{proof}

\begin{proof}[Proof of Lemma \ref{l6}]
We apply a classical  compactness-uniqueness argument. Suppose by
contradiction that (\ref{eq26}) does not hold. Then there exists a
sequence $(v_{0}^{n},v_{1}^{n})$ of initial data such that the
corresponding solution $v^n$ verifies
$$\frac{||v^n(0)||_{L^2(\Omega)}^{2}+
||v^n(T)||_{L^2(\Omega)}^{2}}{\int_{0}^{T} \int_{\Gamma_0}(x\cdot
\nu)\Big(\frac{\p v^n}{\p \nu}\Big)^2 d\sigma dt}\rightarrow
\infty.$$ Normalizing we may suppose that
\begin{equation}\label{eq92}
||v^{n}(0)||_{L^2(\Omega)}^2+||v^{n}(T)||_{L^2(\Omega)}^{2}=1, \quad
\int_{0}^{T} \int_{\Gamma_0}(x\cdot \nu)\Big(\frac{\p v^n}{\p
\nu}\Big)^2 d\sigma dt\rightarrow 0.
\end{equation}
 From (\ref{eq25}) we deduce that the corresponding energy is
 uniformly bounded. In particular,  we deduce that
 $v^n$ is uniformly bounded in
 $$C([0,T]; H_{\lambda(N)})\cap C^1([0,T]; L^2(\Omega)).$$

 Therefore, by extracting a subsequence
\begin{equation}\label{eq91}
v^{n}\rightharpoonup v \textrm{ in } L^\infty(0,T;H_{\lambda(N)})
\textrm{ weakly-}\star,
\end{equation}
From Theorem \ref{trace} we obtain
$$\frac{\p v^n}{\p \nu}\sqrt{x\cdot \nu}
\rightharpoonup \frac{\p v}{\p \nu} \sqrt{x\cdot \nu} \textrm{ in }
 L^\infty(0,T;L^2(\Gamma_0)) \textrm{ weakly-}\star.$$
Furthermore, by lower semicontinuity,
$$0\leq \int_{0}^{T} \int_{\Gamma_0}(x\cdot \nu)
\Big(\frac{\p v}{\p \nu}\Big)^2 d\sigma dt \leq
\liminf_{n\rightarrow \infty} \int_{0}^{T} \int_{\Gamma_0}(x\cdot
\nu)\Big(\frac{\p v^n}{\p \nu}\Big)^2
 d\sigma dt=0.$$
Hence $$\int_{0}^{T}\int_{\Gamma_0}(x\cdot \nu)\Big(\frac{\p v}{\p
\nu}\Big)^2d\sigma dt=0,$$ and
\begin{equation}\label{eq27}
(x\cdot \nu)\frac{\p v}{\p\nu}=0, \quad \textrm{ a.e. on }
\Gamma_0,\quad \forall t\in [0,T].
\end{equation}
On the other hand, from compactness we deduce that
$$v^{n}\rightarrow v \textrm{ in } L^\infty(0,T;L^2(\Omega)),$$
 which combined with (\ref{eq92}) yield to
\begin{equation}\label{eq28}
||v_0||_{L^2(\Omega)}^{2}+||v(T)||_{L^2(\Omega)}^{2}=1.
\end{equation}
To end the proof of Lemma \ref{l6} it suffices to observe that
(\ref{eq27})-(\ref{eq28}) lead to a contradiction. Indeed, in view
of (\ref{eq27}) and by Holmgreen unique continuation we deduce that
$v\equiv 0$ in $\Omega$ which is in contradiction with (\ref{eq28}).
\end{proof}

\begin{obs}
Unique continuation results may be applied far from origin where
coefficient of the lower order term of the operator
$-\p_{tt}-\D-\la/|x|^2$ is analytic in time (actually, it is
independent of time and bounded in space). The principal part
coincides with the D'Alambertian operator, then one can apply
Homlgreen's unique continuation to get $v=0$ a.e. in $\Omega
\setminus B(0,\eps)$ for any $\eps>0$. In consequence, we will have
$v\equiv 0$ in $\Omega$, see \cite{Tataru1}.
\end{obs}
\subsection{The Schr\"{o}dinger equation}\label{6sec}

In this section we consider the Schr\"{o}dinger-like equation

\begin{equation}\label{Seq124Sb}\left\{\begin{array}{ll}
  iu_{t}-\Delta u-\lambda \frac{u}{|x|^2}=0, & (t,x)\in Q_T, \\
  u(t,x)=h(t,x), & (t,x)\in (0,T)\times \Gamma_0, \\
  u(t,x)=0, & (t,x)\in (0,T)\times (\Gamma\setminus \Gamma_0), \\
  u(0,x)=u_0(x), &  x\in \Omega, \\
\end{array}\right.
\end{equation}
Moreover,  we assume $\Omega\subset \rr^N$, $N\geq 1$, is a smooth
bounded domain satisfying case C1 and $\la\leq \la(N):=N^2/4$.
 For the Schr\"{o}dinger equation we define the Hilbert
spaces $L^2(\Omega; \mathbb{C})$ and $H_0^1(\Omega; \mathbb{C})$
endowed with the inner products
$$<u, v>_{L^2(\Omega; \mathbb{C})}:=\textrm{Re}\into u(x)\overline{v(x)}dx, \q
\forall u, v\in L^2(\Omega; \mathbb{C}), $$
$$<u, v>_{H_0^1(\Omega; \mathbb{C})}:=\textrm{Re}\into \n u(x)\cdot \n \overline{v(x)}
dx, \q \forall u, v\in H_0^1(\Omega; \mathbb{C}).$$
 For all $\la \leq
\la(N)$,  we also define the Hilbert space $H_\la (\Omega;
\mathbb{C})$ as the completion of $H_0^1 (\Omega; \mathbb{C})$ with
respect to the norm associated with the inner product
\begin{equation}\label{Scomplexproduct}
<u, v>_{H_\la(\Omega; \mathbb{C})}:=\textrm{Re}\into \big(\n
u(x)\cdot \n \overline{v(x)}-\la
\frac{u(x)\overline{v(x)}}{|x|^2}\big) dx, \q \forall u, v\in
H_0^1(\Omega; \mathbb{C}).
\end{equation}
The spaces $L^2(\Omega; \mathbb{C})$, $H_0^1(\Omega; \mathbb{C})$,
$H_\la(\Omega; \mathbb{C})$ inherit the properties of the
corresponding real spaces. In order to simplify the notations, we
will write $L^2(\Omega), \hoi, H_\la $ 
without making confusions.

As shown for the wave equation, the system \eqref{Seq124Sb} is well
posed.

\begin{Th}[see \cite{judith}]\label{SWt2}
 Let $T>0$ be given and assume $\lambda\leq \la(N)$. For every
$u_0\in H_{\lambda}^{'}$ and any $h\in L^2((0,T)\times \Gamma_0)$
the system \eqref{Seq124Sb} is  well-posed, i.e. there exists a
unique weak solution such that
\begin{equation*}\label{eq20}
u\in C([0,T];  H_{\lambda}^{'}).
\end{equation*}
Moreover, there exists constant $C>0$ such that the solution of
$(S_\la)$ satisfies
\begin{equation*}\label{eq21}
||u||_{L^\infty(0,T;  H_{\lambda}^{'})}\leq C(||u_0||_{
H_{\lambda}^{'}}+||h||_{L^2((0,T)\times \Gamma_0)}).
\end{equation*}
\end{Th}

The system \eqref{Seq124Sb} is also controllable.   More precisely,
the control result states as follows.

\begin{Th}\label{SSht1}
The system $(S_\la)$ is  controllable for any $\lambda\leq \la(N)$.
More precisely,  for any time $T>0$, $u_0\in H_\la^{'}$ and
$\overline{u_0} \in  H_{\lambda}^{'}$ there exists $h\in
L^2((0,T)\times\Gamma_0)$ such that the solution of ($S_\la$)
satisfies
\begin{equation*}u(T,x)=
\overline{u_0}(x) \quad \textrm{ for all } x\in \Omega.
\end{equation*}
\end{Th}

As discussed in Subsection \ref{9sec}, the controllability  is
equivalent to the Observability inequality for the solution of the
adjoint system
\begin{equation}\label{Seq2}\left\{\begin{array}{ll}
  i v_{t}+\Delta v+\lambda \frac{v}{|x|^2}=0, & (t,x)\in Q_T, \\
  v(t,x)=0, & (t,x)\in (0,T)\times \Gamma, \\
  v(0,x)=v_0(x), &  x\in \Omega, \\
\end{array}\right.
\end{equation}
More precisely, if $v$ solves \eqref{Seq2}, then for any time $T>0$,
there exists a positive constant $C_T$ such that
\begin{equation}\label{SSbObsIneq}
||v_0||_{H_\la}^{2} \leq C_T \int_{0}^{T}\int_{\Gamma_0}(x\cdot
\nu)\Big|\frac{\p v}{\p \nu}\Big|^2d\sigma dt.
\end{equation}
 Observability \eqref{SSbObsIneq} might be deduced
directly using the multiplier identity stated in Lemma \ref{Smult}.
The proof is let to the reader since it follows the same steps in
\cite{judith}.

\begin{lema}\label{Smult}
Assume $\la \leq \la_\star$ and $v$ is the solution of \eqref{Seq2}
corresponding to the initial data $v_0\in H_\la$. Then
\begin{equation}
\int_{0}^{T}\int_{\Gamma}  \Big|\frac{\p v }{\p \nu}\Big|^2|x|^2
d\sigma dt \lesssim ||v_0||_{H_\la}^{2}
\end{equation}
and $v$ satisfies the identity
$$\frac{1}{2}\int_{0}^{T}\int_{\Gamma}
(x\cdot \nu) \Big|\frac{\p v }{\p \nu}\Big|^2d\sigma dt
=T||v||_{H_\la}^{2} + \frac{1}{2}\textrm{Im}\int_{\Omega} v x\cdot
\nabla \overline{v} dx\Big|_{t=0}^{t=T}.$$
\end{lema}

\begin{obs}
  Besides, the proof of \eqref{SSbObsIneq} can be deduced from
   the result valid for the wave equation.
Indeed, the general theory presented in an abstract form in
\cite{tuch},  assure the observability of systems like $\dot
z=iA_0z$ using results available for systems of the form
$\ddot{z}=-A_0 z$.
\end{obs}
\section{Open problems}\label{4sec}

\noindent {\bf 1. Geometric constraints.}  In this paper we have
shown the role of the Pohozaev identity, in the context of boundary
singularities,   when  studying  the controllability of conservative
systems like Wave and Schr\"{o}dinger equations. We proved that for
any $\la\leq \la(N)=N^2/4$,  the corresponding systems are exact
observable from $\Gamma_0$ precised in \eqref{oeq2}. Our result
enlarges the range of values $\la\leq (N-2)^2/4$  for which the
control holds, proved firstly in \cite{judith} in the context of
interior singularities.

The geometrical assumption for $\Gamma_0$ is really necessary,
otherwise our proof does not work. Of course, it is still open to be
analysed the case when the central of gravity of $\Gamma_0$ is
centered at a point $x_0$ different by zero, i.e. $\Gamma_0=\{x \in
\Gamma \ |\ (x-x_0)\cdot \nu\geq 0\}$. This choice of $\Gamma_0$
provide some technical difficulties which have been also emphasized
in \cite{judith}. En eventually proof in the case of a such  domain
$\Gamma_0$ should apply a different technique that we have used so
far.

\noindent {\bf 2.  Multipolar singularities}. The same Pohozaev
identity and controllability issues could be address for more
complicated operators, like for instance $L=-\D -V(x)$, where $V(x)$
denotes a multi-particle potential. To the best of our knowledge,
even if there are some important works studying Hardy-type
inequalities for multipolar potentials (see e.g. \cite{MR2379440} et
al.), an accurate analysis is still  to be done. In a forthcoming
work we study two particles systems. Our goal is to analyze the
limit process when one particle collapses to the other. We apply
this both in the context of controllability and the diffusion heat
processes discussing the time decay of solutions.

\section{Appendix: sharp bounds for $||x\cdot \n
v(t)||_{L^2(\Omega)}$}\label{7sec}

\begin{proof}[Proof of Theorem \ref{tu8}]
Without losing the generality it is enough to consider two type of
geometries for $\Omega$ as follows.

G1: The points on $\Gamma$ satisfy $x_N\geq 0$ in the neighborhood
of origin.

G2; The points on $\Gamma$ satisfy $x_N<0$ in the neighborhood of
origin.

In the other intermediate case (when $x_N$ changes sign at origin)
the result valid for case G2 still holds true since we can prove it
for test functions  extended with zero up to a domain satisfying G2.

The proof follows several steps. \\

\noindent {\bf Step 1.} Firstly we show that Theorem \ref{tu8} is
locally true. More precisely, there exists $r_0=r_0(\Omega, N)>0$
small enough, and $C=C(r_0)$ such that
\begin{equation}\label{eqr0}\int_{\Omega_{r_0}}|x|^2|\nabla v|^2dx\leq
R_{\Omega}^{2}\Big[\int_{\Omega_{r_0}}|\nabla
v|^2dx-\frac{N^2}{4}\int_{\Omega_{r_0}}\frac{v^2}{|x|^2}dx\Big]+C(r_0)\int_{\Omega_{r_0}}
v^2dx,
\end{equation}
holds true for any function $v\in C_{0}^{\infty}(\Omega_{r_0})$,
where $\Omega_{r_0}=\Omega \cap B_{r_0}(0)$.

Next we check the validity of Step 1. For that let us consider a
function $\phi$ which satisfies
$$-\D \phi\geq \frac{N^2}{4}\frac{\phi}{|x|^2}, \q \phi >0,  \q \forall x\in \Omega_{r_0},$$
for some positive constant $r_0$. Such a function exists for each
one of the case G1-G2. Indeed, for the case G1 we may consider
$\phi=x_N |x|^{-N/2}$ and for case G2 we can take\\
 $\phi=
d(x)e^{(1-N)d(x)}\Big|\log \frac{1}{|x|}\Big|^{1/2}|x|^{-N/2}$.

Next we introduce  $u$ such that $v=\phi u$. Then we get
$$|\n v|^2 = |\n \phi|^2 u^2 +\phi^2 |\n u|^2 +2 \phi u \n \phi \cdot \n u . $$
Next,  integrating we get
\begin{align}\label{lema2}
\intos |\n v|^2dx &
= \intos |\n u|^2 \phi^2 dx-\intos \frac{\D \phi}{\phi} v^2.
\end{align}
On the other hand, we obtain
\begin{align}\label{eqp1}
\intos |x|^2 |\n v|^2 dx=\intos |x|^2 |\n \phi|^2 u^2 dx +\intos
|x|^2 \phi^2 |\n u|^2 dx+\frac{1}{2}\intos |x|^2 \n (\phi^2) \cdot
\n (u^2) dx
\end{align}
Next we deduce
\begin{align}\label{eqp2}
\frac{1}{2}\intos |x|^2 \n(\phi^2)\cdot \n (u^2) dx&
=-\intos 2 \frac{x\cdot \n \phi}{\phi} v^2dx -\intos |x|^2 |\n
\phi|^2u^2 dx-\intos \frac{ \D \phi}{\phi} |x|^2 v^2 dx .
\end{align}
According to \eqref{eqp1} and \eqref{eqp2} we obtain
\begin{align}\label{eqp3}
\intos |x|^2 |\n v|^2 dx&= \intos |x|^2 \phi^2 |\n u|^2 dx -\intos
\frac{ 2 x \cdot \n \phi}{\phi} v^2dx -\intos  \frac{\D \phi}{\phi}
 |x|^2 v^2 dx.
\end{align}

Let us write
\begin{align}\label{suma}
-\frac{\D \phi}{\phi}= \frac{N^2}{4|x|^2}+P,
\end{align}
where $P\geq 0$ for any $x\in \Omega_{r_0}$. Then from \eqref{lema2}
we have
\begin{align}\label{integr}
\intos |x|^2 \phi^2 |\n u|^2 dx &\leq R_\Omega^2\intos \phi^2 |\n
u|^2= R_\Omega^2  \Big[\intos |\n v|^2 dx+\intos \frac{\D
\phi}{\phi}
v^2 \Big]\nonumber\\
&= R_\Omega^2 \intos \Big[ |\n v|^2-\frac{N^2}{4}\frac{v^2}{|x|^2}
\Big]dx -R_\Omega^2\intos Pv^2dx.
\end{align}
From above and \eqref{eqp3} it follows that
\begin{align}\label{eqp4}
\intos |x|^2 |\n v|^2dx &\leq R_\Omega^2 \intos \Big[ |\n v|^2
-\frac{N^2}{4}\frac{v^2}{|x|^2}\Big]dx - R_\Omega^2 \intos P
v^2dx\nonumber\\
& -2\intos \frac{x\cdot \n \phi}{\phi} v^2 dx + \intos \Big(
\frac{N^2}{4|x|^2}+P\Big)|x|^2 v^2 dx\nonumber\\
& = R_\Omega^2 \intos \Big[ |\n v|^2
-\frac{N^2}{4}\frac{v^2}{|x|^2}\Big]dx +\intos(|x|^2-R_\Omega^2) P
v^2dx\nonumber\\
&-2\intos \frac{x\cdot \n \phi}{\phi} v^2 dx + \frac{N^2}{4} \intos
v^2 dx.
\end{align}
In the case G1 for $r_0$ small enough we have $P=0$ and
   $$\Big|\frac{x\cdot \n \phi}{\phi}\Big|\leq C, \q \forall x\in \Omega_{r_0},$$
 holds   for some positive constant $C$. Thanks to \eqref{eqp4}
   we conclude the proof of Step 1 in the cases G1.

   In the case G2, for $r_0$ small enough we have
   $$P>0, \q \n d\cdot x\geq 0, \q \forall x\in \Omega_{r_0}$$
Then, we remark
$$\frac{x\cdot \n \phi}{\phi}= \frac{x\cdot \n d}{d} +O(1),$$ and
from above we also finish the proof of Step 1 in this case.

\noindent {\bf Step 2.}  This step consist in applying a cut-off
argument to transfer the validity of inequality \eqref{eqr0} from
$\Omega_{r_0}$ to $\Omega$. More precisely, we consider a cut-off
function $\theta\in C_{0}^{\infty}(\Omega)$ such that
\begin{equation}\theta(x)=\left\{\begin{array}{ll}
  1, & |x|\leq r_0/2, \\
  0, & |x|\geq r_0. \\
\end{array}\right.
\end{equation}
Then we split $v\in C_{0}^{\infty}(\Omega)$ as follows,
$$v=\theta v+(1-\theta) v:= w_1+w_2.$$

Next let us firstly prove the following lema.

\begin{lema}\label{l9}
Let us consider a weight function $\rho:
C^{\infty}(\overline{\Omega})\rightarrow \rr$ which is bounded and
non negative. There exists $C(\Omega,\rho)>0$ such that the
following inequality holds
\begin{equation}\label{eq102}
\int_{\Omega}\rho(x)\nabla w_1\cdot\nabla w_2 dx \geq
-C(\Omega,\rho,r)\int_{\Omega}|v|^2dx.
\end{equation}
\end{lema}

\begin{proof}[Proof of Lemma \ref{l9}]
From the boundary conditions, integrating by parts we have
\begin{align}\label{eq103}
\into\rho \nabla w_1\cdot\nabla w_2dx
&=\int\rho\theta(1-\theta)|\nabla v|^2dx+\into \rho v\nabla v\cdot
\nabla \rho(1-2\theta)dx-\int\theta|\nabla
\theta|^2|v|^2dx\nonumber\\
&\geq  \frac{1}{2}\int_{\Omega_{r_0}\setminus\Omega_{r_0/2}}\nabla
(|v|^2)\cdot \nabla \theta (1-2\theta)\rho dx
-||\rho||_{\infty}||D\theta||_{\infty}^{2}\int_{\Omega}|v|^2dx\nonumber\\
&=-\frac{1}{2}\int_{\Omega_{r_0}\setminus\Omega_{r_0/2}}\textrm{div}((1-2\theta)\rho\nabla
\theta)|v|^2dx-||\rho||_{\infty}||D\theta||_{\infty}^{2}
\int_{\Omega}|v|^2dx\nonumber\\
&\geq
-C(||\rho||_{W^{1,\infty}},||\theta||_{W^{2,\infty}})\int_{\Omega}|v|^2dx.
\end{align}
\end{proof}

Now we are able to finalize the proof of Step 2. Indeed, splitting
$v$ as before we get
\begin{equation*}
\int_{\Omega}|x|^2|\nabla v|^2dx
= \int_{\Omega_{r_0}}|x|^2|\nabla
w_1|^2dx+\int_{\Omega\setminus\Omega_{r_0/2}} |x|^2|\nabla
w_2|^2dx+2\int_{\Omega_{r_0}\setminus\Omega_{r_0/2}} |x|^2 \nabla
w_1\cdot\nabla w_2dx
\end{equation*}
 Applying \eqref{eqr0} to $w_1$ we obtain
\begin{align}\label{equ42}
\int_{\Omega}|x|^2|\nabla v|^2dx 
 &\leq R_{\Omega}^{2}\Big[\int_{\Omega} |\nabla v|^2dx-\frac{N^2}{4}
 \into \frac{w_1^2}{|x|^2}dx\Big]+C\int_{\Omega} v^2dx-\nonumber\\
 &-\int_{\Omega_{r_0}\setminus\Omega_{r_0/2}} 2(R_{\Omega}^{2}-|x|^2)\nabla w_1\cdot\nabla
w_2dx.
\end{align}

Adding $\rho=2(R_{\Omega}^{2}-|x|^2))$ in Lemma \ref{l9}, from
(\ref{equ42}) we get
\begin{equation}\label{eq104}
\into|x|^2|\nabla v|^2dx\leq R_{\Omega}^{2}\Big[\into |\nabla
v|^2dx-\frac{N^2}{4}\int_{\Omega_{r_0}}
\frac{w_1^2}{|x|^2}dx\Big]+C(\Omega,r_0)\into v^2dx.
\end{equation}
On  the other hand we remark that
\begin{align}\label{equ43}
\int_{\Omega_{r_0}}\frac{w_1^2}{|x|^2} &
\geq \into \frac{v^2}{|x|^2}dx-C(r_0)\into v^2dx.
\end{align}
From (\ref{eq104}) and (\ref{equ43}) the  conclusion of Theorem
\ref{tu8} yields choosing  $r_0$ small enough, $r_0\leq R_\Omega$.
\end{proof}

\noindent {\bf Acknoledgements} The author wish to thank   Enrique
Zuazua and Adimurthi for useful suggestions and advices.

 Partially supported by the Grants MTM2008-03541 and
 MTM2011-29306-C02-00 of the MICINN (Spain), project
  PI2010-04 of the Basque Government, the ERC Advanced Grant FP7-246775
  NUMERIWAVES, the ESF Research Networking
Program OPTPDE, the grant PN-II-ID-PCE-2011-3-0075 of CNCS-UEFISCDI
Romania and a doctoral fellowship from UAM (Universidad Aut\'{o}noma
de Madrid).

\end{document}